\documentclass[hidelinks,onefignum,onetabnum]{siamart251216}

\usepackage{lipsum}
\usepackage{amsfonts}
\usepackage{graphicx}
\usepackage{epstopdf}
\usepackage{comment}
\usepackage{bm}
\usepackage{algorithmic}
\ifpdf
  \DeclareGraphicsExtensions{.eps,.pdf,.png,.jpg}
\else
  \DeclareGraphicsExtensions{.eps}
\fi

\usepackage{amssymb}
\usepackage{xcolor}
\usepackage{algorithm}

\usepackage{subcaption}

\renewcommand{\le}{\leqslant}
\renewcommand{\ge}{\geqslant}
\renewcommand{\leq}{\leqslant}
\renewcommand{\geq}{\geqslant}
\renewcommand{\emptyset}{\varnothing}

\newcommand{\real}{\mathbb{R}}

\newcommand{\natu}{\mathbb{N}}

\newcommand{\bsh}{\boldsymbol{h}}

\newcommand{\bsl}{\boldsymbol{\ell}}

\newcommand{\bsx}{\boldsymbol{x}}

\newcommand{\bsone}{\boldsymbol{1}}

\newcommand{\rd}{\,\mathrm{d}}

\newcommand{\supp}{\boldsymbol{s}}

\newsiamremark{remark}{Remark}
\newsiamremark{hypothesis}{Hypothesis}
\crefname{hypothesis}{Hypothesis}{Hypotheses}
\newsiamthm{claim}{Claim}
\newsiamremark{fact}{Fact}
\crefname{fact}{Fact}{Facts}

\headers{Median coarsely scrambled Sobol'}{Ziyang Ye, Chaokun Zhu, and Zexin Pan}

\title{Robust high-dimensional integration using medians of coarsely scrambled Sobol' sequences\thanks{The first author acknowledges the support of the National Natural Science Foundation of China (grant 72571153).}}

\author{Ziyang Ye\thanks{Department of Mathematical Sciences, Tsinghua University, Beijing 100084, People's Repulic of China (yzy23@mails.tsinghua.edu.cn).}\and Chaokun Zhu\thanks{School of Mathematical Sciences, Zhejiang University, 866 Yuhangtang Road, Hangzhou 310058, Zhejiang, People's Repulic of China (\email{ckz@zju.edu.cn}).} \and  Zexin Pan\thanks{Corresponding author. Institute of Fundamental and Transdisciplinary Research, Zhejiang University, 866 Yuhangtang Road, Hangzhou 310058, Zhejiang, People's Repulic of China (\email{zep002@zju.edu.cn}).}
}

\usepackage{amsopn}

\ifpdf
\hypersetup{
	pdftitle={Median coarse},
	pdfauthor={Ziyang Ye, Chaokun Zhu, and Zexin Pan}
}
\fi

\begin{document}

\maketitle

\begin{abstract}
	We study the numerical approximation of high-dimensional integrals using randomized quasi-Monte Carlo (RQMC) methods, with a focus on scrambled Sobol' sequences. While asymptotically faster than Monte Carlo, classical RQMC suffers from error bounds that grow exponentially in the dimension $s$, and its root mean squared error (RMSE) convergence rate is generally no better than $O(N^{-3/2})$. To overcome these limitations, we combine two recent developments: Suzuki's coarse scrambling and the median trick. Coarse scrambling randomizes Sobol' sequences according to their generating base polynomials and significantly reduces the maximal gain coefficient. By appropriately choosing the base polynomials, we show that the coefficient can be made uniformly bounded in $s$, yielding an $O(N^{-1/2})$ RMSE for $L^2$ integrands with a dimension-independent constant. The median trick then enables near-optimal convergence for function classes beyond $L^2$: for $L^p$ integrands with $p\in(1,2)$, the median of independent coarsely scrambled estimates achieves an error of $O(N^{-1+1/p})$ with high probability; for integrands in the Haar wavelet space $\mathcal H_{\mathrm{wav},\alpha,s,p,q}$ with $\alpha_p:=\alpha-(1/p-1/2)_+>0$, we prove a high-probability error bound of $O(N^{-\alpha_p-1/2+\varepsilon})$ for any $\varepsilon>0$; the bound improves to $O(N^{-r-\alpha_p-1/2+\varepsilon})$ when the integrand has dominating mixed derivatives of order $r\ge1$ that belong to $\mathcal H_{\mathrm{wav},\alpha,s,p,q}$. The latter two bounds are uniform in $s$ under suitable conditions on the ANOVA components of the integrands. Numerical experiments confirm the predicted convergence rates and demonstrate the robustness of the proposed approach.
\end{abstract}



\section{Introduction}\label{sec:intro}

We consider the numerical approximation of the integral
\begin{equation*}
I(f):=\int_{(0,1)^s} f(\bm x)\,d\bm x,
\end{equation*}
where $f:(0,1)^s\to\mathbb R$ is Lebesgue integrable. Monte Carlo (MC) approximates $I(f)$ by the sample mean of $f(\bsx)$ at $N$ independent uniform points in $(0,1)^s$. For square-integrable $f$, the root mean square error (RMSE) is $O(N^{-1/2})$, independent of $s$.

To achieve higher convergence rates, quasi-Monte Carlo (QMC) replaces independent samples with low-discrepancy point sets. Standard constructions include lattice rules \cite{dick2022lattice,sloan1994lattice} and digital nets \cite{dick:pill:2010,niederreiter1992random}. In this paper, we focus on digital nets, particularly Sobol' sequences \cite{Sobol1967}. For integrands of bounded variation in the sense of Hardy and Krause, QMC achieves an integration error of \(O(N^{-1}(\log N)^{s-1})\) \cite{hick:2014}. These classical results, however, are typically formulated for integrands on \([0,1)^s\). Since low-discrepancy point sets may contain boundary points, notably the origin, this poses a challenge when the integrand is not defined at the boundary.

Randomized quasi-Monte Carlo (RQMC) methods randomize low-discrepancy point sets to yield unbiased estimates and statistical error assessment for the target integral, while avoiding evaluations at the boundary almost surely. This approach was pioneered by Owen \cite{Owen1995} through his nested uniform scrambling. For square-integrable integrands, scrambled digital nets retain the \(O(N^{-1/2})\) RMSE rate, with an implied constant controlled by the maximal gain coefficient \cite{owen1998scrambling}. The rate further improves to \(O(N^{-3/2}(\log N)^{(s-1)/2})\) under regularity conditions \cite{owen1997scrambled}. For integrands with boundary singularities, convergence can be established under Owen's boundary growth condition \cite{owen2006halton}. Later, Matoušek \cite{Matousek1998} proposed random linear scrambling, which maintains the same RMSE as Owen's scrambling while offering simpler implementation. Nevertheless, these methods still suffer from error bounds that grow exponentially in \(s\) \cite{goda2023improved,pan:owen:2023}, and their convergence rates in general do not exceed \(O(N^{-3/2})\) \cite{loh:2003,pan2023superpolynomial}.

To overcome these difficulties, we combine two recent developments: coarse scrambling and the median trick. Coarse scrambling, introduced by Suzuki \cite{suzuki2025coarse} as a generalization of scrambled Halton points \cite{owen2024gain}, reformulates base-\(b\) generalized Niederreiter sequences (GNS) as equidistributed sequences in base \((b^{e_1},\dots,b^{e_s})\), where \(e_j\) is the degree of the base polynomial for coordinate \(j\) (see Section~\ref{subsec:GNS}). Applying scrambling within this base significantly reduces the maximal gain coefficient: for Sobol' sequences, which are base-\(2\) GNS, this yields a reduction from at least \(2^{s-1}\) to \(O(\log s)\).

The second ingredient, the median trick, offers a way to achieve convergence rates beyond \(O(N^{-3/2})\) without requiring user-specified smoothness parameters. Its use in QMC has been explored in a series of works \cite{goda2024simpleuniversal,goda:lecu:2022,goda2026quasi,goda2024universal,pan2026automatic,pan2026dimension,pan2023superpolynomial,pan2024superpolynomial}, where the median of independent RQMC estimates attains near-optimal convergence rates over various function spaces without prior knowledge of the function smoothness.

In this work, we characterize smoothness beyond \(L^2\) using Haar wavelet spaces \cite{gnewuch2026qmc}. These spaces generalize Sobolev spaces of dominating mixed smoothness and unify regularity conditions expressed in terms of fractional derivatives (as in \cite{dick2008koksma,gnewuch2026qmc,pan2026automatic}) or fractional Vitali variation (as in \cite{dick:2008,goda2026quasi,pan2026automatic,pan2026dimension}).

We now summarize our main contributions. In Section~\ref{sec:main}, we introduce the reduced GNS, where the number of degree-\(n\) base polynomials used in the construction is limited according to a prescribed reduction rule (see Definition~\ref{def:reducedGNS}). Section~\ref{subsec:L2} shows that the maximal gain coefficient for such a reduced GNS is uniformly bounded in $s$; consequently, for \(L^2\) integrands, using \(N\) coarsely scrambled points yields an RMSE of \(O(N^{-1/2})\) with a dimension-independent constant.
Section~\ref{subsec:Lp} further shows that for \(L^p\) integrands with \(p\in(1,2)\), the median of \(R\) such estimates achieves an integration error of \(O(N^{-1+1/p})\) with high probability, i.e., the failure probability decays exponentially in \(R\). Section~\ref{subsec:Haarintegrands} studies integrands in the Haar wavelet space \(\mathcal H_{\mathrm{wav},\alpha,s,p,q}\) with \(\alpha_p:=\alpha-(1/p-1/2)_+>0\); the median attains an error of \(O(N^{-\alpha_p-1/2+\varepsilon})\) for any \(\varepsilon>0\) with high probability. Section~\ref{subsec:Haarderiv} considers integrands whose dominating mixed derivatives of order \(r\ge1\) belong to \(\mathcal H_{\mathrm{wav},\alpha,s,p,q}\); the error then improves to \(O(N^{-r-\alpha_p-1/2+\varepsilon})\). The latter two error bounds are uniform in \(s\) when the ANOVA components  of the integrand satisfy the respective  assumptions of Corollaries~\ref{cor:tractability} and~\ref{cor:deriv}.

The remainder of this paper is organized as follows. Section~\ref{sec:prelim} presents the necessary background on digital nets, coarse scrambling, Walsh series, and Haar wavelet spaces. Section~\ref{sec:main} states and proves our main results.
Section~\ref{sec:numerics} numerically compares the proposed method with existing RQMC constructions and standard MC. Section~\ref{sec:discussion} concludes the paper. Technical proofs are deferred to the appendices.

\section{Notation and background}\label{sec:prelim}

For an integer $s\geq 1$, we write
$1{:}s$ for the set $\{1,\dots,s\}$ and $0{:}s$ for $\{0,1,\dots,s\}$. We use $\mathbb Z$ for the
integers, $\natu$ for the positive integers, and $\mathbb N_0$ for the non-negative integers. For any integer
$N\ge1$, we let $\mathbb Z_N:=\{0,1,\dots,N-1\}$. Let $b$ be a prime
power, and denote by $\mathbb F_b$ the finite field with $b$ elements.  If $b$ is prime, we identify $\mathbb F_b$ with $\mathbb Z_b$ with addition,
subtraction, and multiplication defined modulo $b$.
For $e\in\mathbb N$, $\mathrm{GL}(e,\mathbb F_b)$ denotes the group of nonsingular $e$ by $e$ matrices over $\mathbb F_b$. 

For a subset $u\subseteq1{:}s$ and a vector
$\bm{x}=(x_1,\dots,x_s)$, we write $\bm{x}_u=(x_j)_{j\in u}$.
The cardinality of $u$ is denoted by $|u|$. For
$\bm{k}=(k_1,\dots,k_s)\in\mathbb N_0^s$, we write $|\bm{k}|_1:=\sum_{j=1}^s k_j.$
For vectors of the same length, inequalities are understood componentwise;
for example, $\bm j\ge\bm\ell$ means $j_i\ge\ell_i$ for every coordinate $i$.

For $p\in[1,\infty)$ and $f:(0,1)^s \to \mathbb R$, we write $f\in L^p$ if 
$$\|f\|_{L^p}:=\left(\int_{(0,1)^s}|f(\bm x)|^p\,d\bm x\right)^{1/p}<\infty.$$
For $\bm\tau=(\tau_1,\dots,\tau_s)\in\mathbb N_0^s$, we define the mixed partial derivative
\[
f^{(\bm\tau)}:=\frac{\partial^{|\bm\tau|}f}{\partial x_1^{\tau_1}\cdots\partial x_s^{\tau_s}},
\]
where $f^{(\bm 0)}=f$ by convention. For $r\geq 1$, we write $f\in W_{\mathrm{mix}}^{r,p}$ if the weak derivatives $f^{(\bm\tau)}$ exist for all $\bm\tau\in\{0{:}r\}^s$ and belong to $L^p$.

For $h\in\mathbb N_0$, we write its $b$-adic expansion in the sparse form
$h=\sum_{i=1}^{\nu}c_i b^{a_i-1}$ with
$c_i\in\mathbb Z_b\setminus\{0\}$ and
$a_1>\cdots>a_\nu\ge1$, where $\nu=0$ if $h=0$. For $r\in\mathbb N_0$, define
\begin{equation*}
	\mu_r(h)
	:=
	\sum_{i=1}^{\min(r,\nu)}a_i.
\end{equation*}
Thus $\mu_0(h)=0$ and $\mu_1(h)$ is the largest nonzero digit position of $h$.
For $\bm{h}=(h_1,\dots,h_s)\in\mathbb N_0^s$, we write
\begin{equation}\label{eqn:bsmu}
\Vec{\bm{\mu}}_r(\bm{h})
:=
(\mu_r(h_1),\dots,\mu_r(h_s)) \quad\text{and}\quad
\mu_r(\bm{h}):=
\sum_{j=1}^s\mu_r(h_j).
\end{equation}

\subsection{Digital nets and block affine matrix scrambling}\label{subsec:GNS}

For later matrix operations, we adopt the following digit-vector notation. For $k\in\mathbb N_0$ with $b$-adic expansion
$k=\eta_0+\eta_1b+\eta_2b^2+\cdots$, we denote its digit vector by
\begin{equation*}
\vec k=(\iota(\eta_0),\iota(\eta_1),\iota(\eta_2),\ldots)^\top
\in\mathbb F_b^{\mathbb N},
\end{equation*}
where $\iota:\mathbb Z_b\to\mathbb F_b$ is a fixed bijection with $\iota(0)=0$. 
Similarly, for $x\in[0,1)$, we choose the expansion
\begin{equation*}
x=\xi_1b^{-1}+\xi_2b^{-2}+\cdots,
\qquad
\xi_r\in\mathbb Z_b,
\end{equation*}
with infinitely many digits different from $b-1$, and set
\begin{equation*}
\vec x=(\iota(\xi_1),\iota(\xi_2),\ldots)^\top\in\mathbb F_b^{\mathbb N}.
\end{equation*}
Conversely, for $\vec z=(z_1,z_2,\ldots)^\top\in\mathbb F_b^{\mathbb N}$, we define
\begin{equation*}
\phi_b(\vec z):=\sum_{r=1}^{\infty}\iota^{-1}(z_r)b^{-r}.
\end{equation*}
All vector and matrix operations involving digit vectors are performed over $\mathbb F_b$.

Using the above notation, a digital sequence $\mathcal S$ generated by matrices
$C_1,\ldots,C_s\in\mathbb F_b^{\mathbb N\times\mathbb N}$ is constructed as
follows. For $n\in\mathbb N_0$, let $\vec n$ be its digit vector. For each $j=1,\dots,s$, set
\begin{equation*}
\vec y_{n,j}:=C_j\vec n
=
(y_{n,j,1},y_{n,j,2},\ldots)^\top
\in\mathbb F_b^{\mathbb N},
\end{equation*}
and define the $j$-th coordinate of the $n$-th point by
\begin{equation*}
x_{n,j}:=\phi_b(\vec y_{n,j})
=\sum_{r=1}^{\infty}\iota^{-1}(y_{n,j,r})b^{-r}.
\end{equation*}
For $m\in\mathbb N$, let
$\mathcal P_m:=\{\bm{x}_0,\ldots,\bm{x}_{b^m-1}\}\subseteq[0,1)^s$.
Equivalently, $\mathcal P_m$ is the digital net formed by the first $b^m$
points of $\mathcal S$, or generated by the first $m$ columns of
$C_1,\ldots,C_s$.

For the reader's convenience, we restate the following definitions from \cite{suzuki2025coarse,tezuka2013discrepancy}.

\begin{definition}[Equidistributed sequence in base $\bm b$]
	Let $\mathcal S=\{\bm x_n\}_{n\geq 0}\subseteq [0,1)^s$ be a sequence of
	points. We call $\mathcal S$ an equidistributed sequence in base
	$\bm b=(b_1,\ldots,b_s)$ if, for any
	$r,k_1,\ldots,k_s\in\mathbb N_0$, every elementary interval of the form 
    \begin{equation*}
\prod_{j=1}^s
\left[
\frac{a_j}{b_j^{k_j}},
\frac{a_j+1}{b_j^{k_j}}
\right),
\qquad
a_j\in\mathbb Z_{b_j^{k_j}} .
\end{equation*}
 contains exactly one point from
	$\{\bm x_{rB},\ldots,\bm x_{(r+1)B-1}\}$, where
	$B=b_1^{k_1}\cdots b_s^{k_s}$.
\end{definition}

\begin{definition}[Generalized Niederreiter sequence, GNS]\label{def:generalized-niederreiter}
	Let $b$ be a prime power and let the base polynomials
	$p_1(x),p_2(x),\ldots,p_s(x)$ be monic polynomials in $\mathbb F_b[x]$
	which are pair-wise coprime. Let $e_j:=\deg(p_j)$. For $k\in\mathbb N$ and
	$1\le j\le s$, let $k_j=\lfloor(k-1)/e_j\rfloor$ and let
	$y_k^{(j)}(x)\in\mathbb F_b[x]$ with the restriction that the residue
	polynomials $y_k^{(j)}(x)\bmod p_j(x)$ for
	$(t-1)e_j\le k-1<te_j$ are linearly independent over $\mathbb F_b$.
	Consider the Laurent series expansion
	\begin{equation*}
	\frac{y_k^{(j)}(x)}{(p_j(x))^t}
	=
	\sum_{r=1}^{\infty}
	\frac{a^{(j)}(t,k,r)}{x^r}
	\in \mathbb F_b((x^{-1})).
	\end{equation*}
	The generalized Niederreiter sequence (GNS) is defined as the digital sequence
	whose generating matrices are $C_j=(c_{k,r}^{(j)})_{k,r\in\mathbb N}$,
	$1\le j\le s$, with
	\begin{equation*}
	c_{k,r}^{(j)}=a^{(j)}(k_j+1,k,r).
	\end{equation*}
\end{definition}

\begin{definition}[Block affine matrix scramble]
	For a prime number $b\ge2$, the block affine matrix scramble in base
	$(b^{e_1},\ldots,b^{e_s})$ is a random mapping from $[0,1)^s$ to $[0,1)^s$ defined as follows. For $j\in 1{:}s$, let
	$D_j\in\mathbb F_b^{\mathbb N}$ and
	$M_j\in\mathbb F_b^{\mathbb N\times\mathbb N}$ be independent random vectors
	and matrices. Each $D_j$ has independent and identically distributed uniform entries in $\mathbb F_b$.
    Each $M_j$ is a block lower-triangular matrix with block size
	$e_j$, where diagonal blocks are drawn uniformly from the nonsingular
	matrices in $\mathrm{GL}(e_j,\mathbb F_b)$, and off-diagonal blocks are drawn
	uniformly from all matrices in $\mathbb F_b^{e_j\times e_j}$. The scramble then maps $\bm x=(x_1,\dots,x_s)$ to
	\begin{equation*}
	\widetilde{\bm x}
	:=
	\left(
	\phi_b(M_1\vec x_1+D_1),
	\ldots,
	\phi_b(M_s\vec x_s+D_s)
	\right).
	\end{equation*}
    For a sequence $\mathcal{S}\subseteq[0,1)^s$, the scrambled sequence is obtained by applying this map elementwise.
\end{definition}

By \cite[Theorem~1]{tezuka2013discrepancy}, a GNS $\mathcal{S}$ with base polynomials $p_1(x),\dots,p_s(x)$ is equidistributed in base $(b^{e_1},\dots,b^{e_s})$ with $e_j=\deg(p_j)$. We refer to the randomization of $\mathcal{S}$ by the block affine matrix scramble in this base as \textbf{coarse scrambling}.

We note that for each $j\in 1{:}s$, $M_j\vec x_j+D_j$ has the same distribution as $D_j$, and hence $\phi_b(M_j\vec x_j+D_j)$ is uniform on $[0,1)$. Independence then implies that $\widetilde{\bm x}$ is uniform on $[0,1)^s$ for any fixed $\bm x$. Furthermore, the support of coarsely scrambled GNS lies in $(0,1)^s$ almost surely.

\subsection{Walsh expansion and dual net}

For $h\in\mathbb N_0$ and $x\in[0,1)$, write
$h=\eta_0+\eta_1b+\cdots+\eta_{a-1}b^{a-1}$ and
$x=\xi_1b^{-1}+\xi_2b^{-2}+\cdots$, where $\eta_r,\xi_r\in\mathbb Z_b$ and
$\eta_r=0$ for all sufficiently large $r$. The $h$-th $b$-adic Walsh function
is
\begin{equation*}
\operatorname{wal}_h(x)
:=
\omega_b^{\sum_{r=0}^{a-1}\eta_r\xi_{r+1}},
\qquad
\omega_b:=\exp(2\pi\sqrt{-1}/b).
\end{equation*}
For $\bm h=(h_1,\ldots,h_s)$, set
$\operatorname{wal}_{\bm h}(\bm x):=\prod_{j=1}^s\operatorname{wal}_{h_j}(x_j)$.
The system $\{\operatorname{wal}_{\bm h}:\bm h\in\mathbb N_0^s\}$ is a
complete orthonormal basis of the $L^2$ function space. Hence, for $f\in L^2$, we have the Walsh expansion in the $L^2$-sense
\begin{equation}\label{eqn:Walshexpansion}
	f(\bm x)=\sum_{\bm h\in\mathbb N_0^s}
	\widehat f(\bm h)\operatorname{wal}_{\bm h}(\bm x)\quad \text{for}\quad
	\widehat f(\bm{h})=\int_{(0,1)^s}f(\bm{x})\overline{\operatorname{wal}_{\bm{h}}(\bm{x})}\,d\bm{x}.
\end{equation}

The next lemma characterizes the RMSE of coarsely scrambled RQMC estimates in terms of the Walsh coefficients.

\begin{lemma}\label{thm:walsh-error-scrambled-net}
	Let 
	$\widetilde{\mathcal P}_m
	=\{\widetilde{\bm x}_0,\ldots,\widetilde{\bm x}_{b^m-1}\}$
	be the first $b^m$ points of the coarsely scrambled
	GNS $\mathcal S$, whose generating matrices are $\bm C=(C_1,\ldots,C_s)$. Let 
	$\bm D=(D_1,\ldots,D_s)$ and $\bm M=(M_1,\ldots,M_s)$ be the random shifts and
	random matrices used in the scrambling. Define the random dual net
	\begin{equation*}
	D_{\bm{C},m}^\perp(\bm M)
	:=
	\left\{
	\bm h\in\mathbb N_0^s:
	\sum_{j=1}^s C_j(1:m,1:m)^\top M_j^\top\vec h_j=\bm0
	\right\}.
	\end{equation*}
	Then for every $f\in L^2$, taking expectation with respect to $\bm D$ for a fixed $\bm M$ gives
	\begin{equation*}
	\mathbb E_{\bm D}
	\left[
	\left|
	\frac1{b^m}
	\sum_{n=0}^{b^m-1}f(\widetilde{\bm x}_n)
	-
	I(f)
	\right|^2	
	\right]
	=
	\sum_{\bm0\neq\bm h\in D_{\bm{C},m}^\perp(\bm M)}
	|\widehat f(\bm h)|^2.
	\end{equation*}
\end{lemma}

The proof is given in Appendix~\ref{app:proof-walsh-error-scrambled-net}. We note that similar results have been established in the literature (see, for instance, \cite{dick:2011,goda2026quasi,pan2024superpolynomial}). However, our proof does not require pointwise convergence of the Walsh expansion \eqref{eqn:Walshexpansion} and thus also applies to discontinuous integrands.

\subsection{Haar wavelet spaces}\label{subsec:haar-space}

We briefly recall the construction of the $b$-adic Haar wavelet system and the associated Haar wavelet spaces; see, e.g., \cite{gnewuch2026qmc} for a comprehensive introduction.

Let $\varphi(x)$ be the indicator $\mathbf1\{0\le x<1\}.$ Fix an integer base $b\ge2$ and define
\[
\Delta_{-1}:=\{0\},
\qquad
\Delta_\ell:=\{0,\ldots,b^\ell-1\},
\quad
\ell\in\mathbb N_0.
\]
For $\ell\in\mathbb N_0$ and $k\in\Delta_\ell$, define
\[
\varphi_k^\ell(x)
=
b^{\ell/2}\varphi(b^\ell x-k),
\]
and let
\[
V^\ell
=
\operatorname{span}
\{\varphi_k^\ell:k\in\Delta_\ell\}.
\]
Then
\[
V^0\subseteq V^\varphi(b^\ell x-k)1\subseteq\cdots\subseteq L^2.
\]
Denote by $ P_\ell:L^2\rightarrow V^\ell$ the orthogonal projection onto $V^\ell$, and adopt the convention $P_{-1}=0. $
For $\ell\ge1$, $i\in \{0,\ldots,b-1\}$, and $k\in\Delta_{\ell-1}$, define
\begin{equation}\label{formula_Haar}
	\psi_{i,k}^\ell(x)=b^{\ell/2}\left(\varphi(b^{\ell} x-bk-i)
	-
	\frac{1}{b}\varphi(b^{\ell-1} x-k)\right)
    .
\end{equation}
Furthermore, let
\[
\psi_{0,0}^0:=\varphi,
\qquad
\nabla_\ell:=
\begin{cases}
	\{0\}, & \ell=0,\\
	\{0,\ldots,b-1\}, & \ell\ge1.
\end{cases}
\]

The $b$-adic Haar wavelet system satisfies the reconstruction formula
\[
(P_\ell-P_{\ell-1})f
=
\sum_{i\in\nabla_\ell}
\sum_{k\in\Delta_{\ell-1}}
\langle
f,\psi_{i,k}^\ell
\rangle
\psi_{i,k}^\ell,
\qquad
\ell\in\mathbb N_0,
\]
and the Parseval identity
\[
\|(P_\ell-P_{\ell-1})f\|_{L^2}^2
=
\sum_{k\in\Delta_{\ell-1}}
\sum_{i\in\nabla_\ell}
\left|
\langle
f,\psi_{i,k}^\ell
\rangle
\right|^2.
\]

We now pass to the $s$-variate tensor-product notation. For
$\bm{\ell}=(\ell_1,\ldots,\ell_s)\in\mathbb N_0^s$, define the product set
\[
\Delta_{\bm{\ell}-\bm1}
:=
\Delta_{\ell_1-1} \times \cdots \times \Delta_{\ell_s-1},
\qquad
\nabla_{\bm{\ell}}
=
\nabla_{\ell_1} \times \cdots \times \nabla_{\ell_s},
\]
and
\[
\psi_{\bm i,\bm k}^{\bm{\ell}}(\bm x)
=
\prod_{j=1}^s
\psi_{i_j,k_j}^{\ell_j}(x_j).
\]

Furthermore, define the Haar detail projection
\[
Q_{\bm{\ell}}
:=
\bigotimes_{j=1}^s
(P_{\ell_j}-P_{\ell_j-1}).
\]
Then
\begin{equation}\label{formula_Q}
	Q_{\bm{\ell}}f
	=
	\sum_{\bm k\in\Delta_{\bm{\ell}-\bm1}}
	\sum_{\bm i\in\nabla_{\bm{\ell}}}
	\left\langle
	f,
	\psi_{\bm i,\bm k}^{\bm{\ell}}
	\right\rangle
	\psi_{\bm i,\bm k}^{\bm{\ell}}.
\end{equation}

We now define the Haar wavelet spaces needed for the subsequent analysis.
\begin{definition}[Haar wavelet space \cite{gnewuch2026qmc}]\label{def:haar-wavelet-space}
	Let $ p,q\in [1,\infty)$ and $\alpha\in\mathbb R$. For
	$f\in L^1$, define
	\begin{equation}\label{eq:haar-wavelet-norm}
		\|f\|_{\mathrm{wav},\alpha,s,p,q}
		:=
		\left[\sum_{\bm{\ell}\in\mathbb N_0^s}
		b^{q(\alpha-1/p+1/2)|\bm{\ell}|_1}
		\left(
		\sum_{\bm k\in\Delta_{\bm{\ell}-\bm 1}}
		\sum_{\bm i\in\nabla_{\bm{\ell}}}
		\left|
		\left\langle f,\psi_{\bm i,\bm k}^{\bm{\ell}}\right\rangle
		\right|^p
		\right)^{q/p}\right]^{1/q}.
	\end{equation}
		The Haar wavelet space is defined by
	\[
	\mathcal H_{\mathrm{wav},\alpha,s,p,q}
	:=
	\left\{
	f\in L^1:
	\|f\|_{\mathrm{wav},\alpha,s,p,q}<\infty
	\right\}.
	\]
\end{definition}
The above definition extends to $p=\infty$ or $q=\infty$ with slight modification; see \cite[Definition 1]{gnewuch2026qmc}. As shown in \cite[Section 2.2]{gnewuch2026qmc},
\begin{equation}\label{eqn:Hnormpq}
\|f\|_{\mathrm{wav},\alpha,s,p,q}\leq \|f\|_{\mathrm{wav},\alpha,s,p',q'} \quad \text{if } p\leq p' \text{ and } q\geq q'.    
\end{equation}

We collect the following two lemmas for later use. Their proofs are given in Appendices~\ref{app:proof-lem:HinL2} and~\ref{app:proof-HaarWalshBlock}, respectively.
\begin{lemma}\label{lem:HinL2}
For $f\in \mathcal H_{\mathrm{wav},\alpha,s,p,q}$, define $\alpha_p:=\alpha-\left(1/p-1/2\right)_+$, where $(x)_+$ denotes $\max(x,0)$.
Then 
\begin{equation}\label{Walsh_bound}
\|f\|_{\mathrm{wav},\alpha_p,s,2,q}^{q}=\sum_{\bm{\ell}\in\mathbb N_0^s}
	b^{q\alpha_p|\bm{\ell}|_1}
	\|Q_{\bm{\ell}}f\|_{L^2}^q
	\le
	\|f\|_{\mathrm{wav},\alpha,s,p,q}^{q}.
\end{equation}
    If further $\alpha_p>0$, there is a constant $C_{b,\alpha_p,q}$ depending on $b,\alpha_p$, and $q$, such that
    \begin{equation}\label{eqn:L2embedding}
      \|f\|_{L^2}\leq C^s_{b,\alpha,p,q} \|f\|_{\mathrm{wav},\alpha,s,p,q}.  
    \end{equation}
\end{lemma}

\begin{lemma}\label{lem:HaarWalshBlock}
	For $\bm{\ell}\in\mathbb N_0^s$, define the Walsh block
	 $\Lambda_{\bm{\ell}}
	:=
	\{\bm h\in\mathbb N_0^s:
	\Vec{\bm{\mu}}_1(\bm{h})=\bm{\ell}\}$
	. Then
	\[
	Q_{\bm{\ell}}f(\bm x)
	=
	\sum_{\bm h\in \Lambda_{\bm{\ell}}}
	\widehat f(\bm h)
	\operatorname{wal}_{\bm h}(\bm x).
	\]
	Consequently,
	\[
	\sum_{\bm h\in \Lambda_{\bm{\ell}}}
	|\widehat f(\bm h)|^2
	=
	\|Q_{\bm{\ell}}f\|_{L^2}^2
	=
	\sum_{\bm k\in\Delta_{\bm{\ell}-\bm1}}
	\sum_{\bm i\in\nabla_{\bm{\ell}}}
	\left|
	\left\langle
	f,
	\psi_{\bm i,\bm k}^{\bm{\ell}}
	\right\rangle
	\right|^2 .
	\]
\end{lemma}

\section{Main results}\label{sec:main}

In this section, we introduce the median estimator based on the reduced GNS and establish its convergence in various function spaces.

\begin{definition}[Reduced GNS]\label{def:reducedGNS}
Given a positive sequence $a(n)$ starting from $n=2$, 
	we say that a GNS is
	$a(n)$-reduced if, for each degree $n\ge2$, at most
	$\lfloor (b^n-1)/(na(n))\rfloor$ monic irreducible polynomials of
	degree $n$ are used as base polynomials. 
\end{definition}

This is in contrast to the full Niederreiter sequence \cite[Definition 2.5]{suzuki2025coarse}, where all monic irreducible polynomials of each degree are used.

Let $\{\widetilde{\bm x}_0,\ldots,\widetilde{\bm x}_{b^m-1}\}$ be the first $b^m$ points from the coarsely scrambled GNS $\mathcal S$. We write the corresponding RQMC estimator as
\begin{equation*}
{\mathcal I}_{m,\mathcal S}(f)
:=
\frac1{b^m}
\sum_{n=0}^{b^m-1}
f(\widetilde{\bm x}_n).
\end{equation*}
Because each $\widetilde{\bm x}_n$ is uniform on $(0,1)^s$, ${\mathcal I}_{m,\mathcal S}(f)$ is unbiased for $I(f)$.

Let $R$ be an odd natural number. For $i\in 1{:}R$, let ${\mathcal I}_{m,\mathcal S}^{(i)}(f)$ be independent realizations of ${\mathcal I}_{m,\mathcal S}(f)$.
The median RQMC estimator is defined by
\begin{equation*}
{\mathcal M}_{m,R,\mathcal S}(f)
:=
\underset{1\leq i\leq R}{\mathrm{median}} \ {\mathcal I}_{m,\mathcal S}^{(i)}(f).
\end{equation*}

We shall repeatedly use the following standard median amplification bound, commonly known as the median trick \cite{kunsch2019solvable,niemiro2009fixed}.

\begin{lemma}\label{lem:median-amplification}
	Let $R\in \natu$ be odd, and let $X_1,\dots,X_R$ be independent copies of a random variable $X$. If for $\mu\in\mathbb R$ and $\varepsilon,\eta>0$,
	\begin{equation*}
	\mathbb P\{|X-\mu|>\varepsilon\}\le\eta,
	\end{equation*}
	then we have
	\begin{equation*}
	\mathbb P\left\{
	\left|\underset{1\le i\le R}{\operatorname{median}}X_i-\mu\right|
	>\varepsilon
	\right\}
	\le
	2^R\eta^{\frac{R+1}{2}}.
	\end{equation*}
\end{lemma}

\begin{proof}
	If the median differs from $\mu$ by more than $\varepsilon$, then at least
	$(R+1)/2$ of the events $\{|X_i-\mu|>\varepsilon\}$ occur. Independence
	and the binomial tail bound give
	\begin{equation*}
	\mathbb P\left\{
	\left|\underset{1\le i\le R}{\operatorname{median}}X_i-\mu\right|
	>\varepsilon
	\right\}
	\le
	\sum_{k=(R+1)/2}^{R}\binom Rk\eta^k
	\le
	2^R\eta^{\frac{R+1}{2}},
	\end{equation*}
	where the last inequality follows from
	$\sum_{k=0}^{R}\binom Rk=2^R$.
\end{proof}

In what follows, all constants may depend on $b$, but we suppress this dependence in the notation for simplicity.

\subsection{Error bounds for $L^2$ integrands}\label{subsec:L2} 
We begin with the notion of the maximal gain coefficient. For \(f\in L^2\), let
$$\sigma^2_{MC}(f):=\int_{(0,1)^s}f(\bm x)^2\,d\bm x-I(f)^2.$$
The maximal gain coefficient of the RQMC estimator \(\mathcal I_{m,\mathcal S}(f)\) is defined as
\begin{equation}\label{eqn:maximalgain}
    \Gamma({\mathcal I}_{m,\mathcal S}):=\sup_{\substack{f\in L^2\\\sigma^2_{MC}(f)\neq 0 }} \frac{\operatorname{Var}\left({\mathcal I}_{m,\mathcal S}(f)\right)}{b^{-m}\sigma^2_{MC}(f)}.
\end{equation}
The terminology is motivated by the fact that \(\operatorname{Var}(\mathcal I_{m,\mathcal S}(f))\) can be expressed as a weighted sum of the variance components of \(f\) in the Haar or Walsh basis, and \(\Gamma(\mathcal I_{m,\mathcal S})\) is precisely the largest multiplier appearing in this sum. We refer the reader to \cite[Section~2.5]{suzuki2025coarse} for further details.

The following lemma shows that $\Gamma({\mathcal I}_{m,\mathcal S})$ is bounded uniformly in  $s$ and $m$ whenever \(\mathcal S\) is a properly reduced GNS.

\begin{lemma}\label{cor:dimension-independent-gain}
	If $\mathcal S$ is a  $(\log n)^{1+\beta}$-reduced GNS with $\beta>0$, then for every $f\in L^2$,
    $$\operatorname{Var}\left({\mathcal I}_{m,\mathcal S}(f)\right)
	\le  \frac{\Gamma({\mathcal I}_{m,\mathcal S})}{b^m}\sigma^2_{MC}(f)
    \le
	 \frac{C_{\beta}}{b^m}\sigma^2_{MC}(f),$$   
where
$$C_\beta=
	\exp\left(
	2+\frac1{2(\log2)^{1+\beta}}
	+\frac1{\beta(\log2)^\beta}
	\right).$$
\end{lemma}

\begin{proof}
By \eqref{eqn:maximalgain}, it suffices to show \(\Gamma(\mathcal I_{m,\mathcal S})\leq C_\beta\). From \cite[Corollary~3.4]{suzuki2025coarse},
\begin{equation*}
	\Gamma({\mathcal I}_{m,\mathcal S})
	\le
	\prod_{j=1}^s
	\frac{b^{e_j}}{b^{e_j}-1}.
	\end{equation*}
	Degree \(1\) contributes at most \(b\) base polynomials to the product. For degree \(n\ge2\), the definition of a \((\log n)^{1+\beta}\)-reduced GNS implies that the number of \(e_j\) equal to \(n\) is at most
	$\lfloor (b^n-1)/(n(\log n)^{1+\beta})\rfloor$. Hence
	\begin{equation*}
	\log\prod_{j=1}^s
	\frac{b^{e_j}}{b^{e_j}-1}
	\le
	b\log\left(\frac b{b-1}\right)
	+
	\sum_{n=2}^\infty
	\frac{b^n-1}{n(\log n)^{1+\beta}}
	\log\left(1+\frac1{b^n-1}\right).
	\end{equation*}
	Since $\log(1+x)\le x$ for $x\geq 0$, we obtain
	\begin{align*}
	\log\prod_{j=1}^s
	\frac{b^{e_j}}{b^{e_j}-1}
	&\le
	\frac{b}{b-1}+
	\sum_{n=2}^\infty
	\frac1{n(\log n)^{1+\beta}}\le
	2+\frac1{2(\log2)^{1+\beta}}
	+
	\int_2^\infty\frac{dx}{x(\log x)^{1+\beta}}\\
	&\le
	2+\frac1{2(\log2)^{1+\beta}}
	+\frac1{\beta(\log2)^\beta}.
	\end{align*}
	Exponentiating both sides proves the claim.
\end{proof}

To bound the RMSE of the median estimator \(\mathcal M_{m,R,\mathcal S}(f)\), we need the following elementary inequality, which controls the variance of a sample median in terms of the variance of the underlying random variable.

\begin{lemma}\label{lem:medvar}
Let $R\in \natu$, and let $X_1,\dots,X_R$ be independent copies of a random variable $X$ with $\mathbb{E} X^2<\infty$. Then
    $$\mathbb{E}\left[\left(\underset{1\leq i\leq R}{\mathrm{median}}  \ X_i-\mathbb{E}X\right)^2\right]\leq 2 \operatorname{Var}(X).$$
\end{lemma}
\begin{proof}
    Assume without loss of generality that \(\mathbb E X=0\). Since \(x\mapsto x^2\) is convex, Jensen's inequality for medians \cite{Merkle2005} gives
     $$\left(\underset{1\leq i\leq R}{\mathrm{median}}  \ X_i\right)^2\leq \underset{1\leq i\leq R}{\mathrm{median}}  \ X^2_i.$$
     Among $X^2_1,\dots,X^2_R$, at least $(R+1)/2$ of them are no smaller than their median if $R$ is odd, and $R/2$ of them if $R$ is even.
    Therefore,
     $$\underset{1\leq i\leq R}{\mathrm{median}}  \ X^2_i\leq \frac{1}{R/2}\sum_{i=1}^R X^2_i.$$
     Taking expectation yields 
     \begin{align*}
          \mathbb{E}\left[\left(\underset{1\leq i\leq R}{\mathrm{median}}  \ X_i\right)^2\right]\leq \mathbb{E}\left[\left(\underset{1\leq i\leq R}{\mathrm{median}}  \ X^2_i\right)\right]\leq \frac{1}{R/2}\sum_{i=1}^R \mathbb{E}X^2_i=2 \mathbb{E}X^2. 
     \end{align*}
\end{proof}

Applying Lemma~\ref{lem:medvar} to $X={\mathcal I}_{m,\mathcal S}(f)$ yields the following bound on ${\mathcal M}_{m,R,\mathcal S}(f)$.

\begin{theorem}\label{thm:l2-median-baseline}
	Under the assumptions of Lemma~\ref{cor:dimension-independent-gain},
	\begin{equation*}
	\mathbb E\left[
	\left|{\mathcal M}_{m,R,\mathcal S}(f)-I(f)\right|^2
	\right]
	\le
	\frac{2C_{\beta}}{b^m} \sigma^2_{MC}(f).
	\end{equation*}
\end{theorem}

\subsection{Error bounds for $L^p$ integrands with $p\in (1,2)$}\label{subsec:Lp}
Although the preceding analysis yields an error bound no better than that of plain MC, we prove that the median estimator ${\mathcal M}_{m,R,\mathcal S}(f)$ enjoys a high‑probability error bound whenever  $f\in L^p$ with $p\in (1,2)$. The proof uses a truncation argument, which serves purely as an analytical tool and is never executed in the actual algorithm.

\begin{theorem}\label{thm:lp-median-bound}
	Let $p\in (1,2)$ and $\delta\in (0,1/8)$. If
	$\mathcal S$ is a $(\log n)^{1+\beta}$-reduced GNS with
	$\beta>0$, then for every $f\in L^p$,
	\begin{align*}
	\mathbb P\left\{
	\left|\mathcal M_{m,R,\mathcal S}(f)-I(f)\right|
	>
	\left(\delta+\sqrt{C_\beta}\right)\delta^{-\frac{1}{p}}
	\|f\|_{L^p}
	b^{-m(1-\frac{1}{p})}
	\right\}
	\le
	\frac{1}{2}(8\delta)^{\frac{R+1}{2}},
	\end{align*}
	where $C_\beta$ is defined in Lemma~\ref{cor:dimension-independent-gain}.
\end{theorem}

\begin{proof}
	Write $N=b^m$ and first consider one replicate
	$\mathcal I_{m,\mathcal S}(f)$. If $\|f\|_{L^p}=0$, the claim is
	immediate. Otherwise, define
	\begin{equation*}
	L_{N,\delta}
	:=
	\delta^{-\frac{1}{p}}\|f\|_{L^p}N^{\frac{1}{p}}
	\end{equation*}
	and the truncated function
	\begin{equation*}
	f_{N,\delta}(\bm x)
	:=
	\max\left\{-L_{N,\delta},
	\min\{f(\bm x),L_{N,\delta}\}\right\}.
	\end{equation*}

	Let $A$ be the event that none of the $N$ function values is truncated.
	By the union bound, Markov's inequality, and the fact that each $\widetilde{\bm x}_n$ is uniform on $(0,1)^s$,
	\begin{align}\label{eq:lp-truncation-sample-event}
	\mathbb P(A^c)
	\le
	\sum_{n=0}^{N-1}
	\mathbb P\{|f(\widetilde{\bm x}_n)|>L_{N,\delta}\}\le
	N L_{N,\delta}^{-p}\|f\|_{L^p}^{p}
	=
	\delta.
	\end{align}
	On $A$, we have
	$\mathcal I_{m,\mathcal S}(f)=
	\mathcal I_{m,\mathcal S}(f_{N,\delta})$.

	Since $|f|\le L_{N,\delta}^{1-p}|f|^p$ on
	$\{|f|>L_{N,\delta}\}$, the truncation bias satisfies
    \begin{equation}\label{eq:lp-truncation-bias}
        \begin{aligned}
            |I(f)-I(f_{N,\delta})|
	&\le
	\int_{(0,1)^s}|f(\bm x)|
	\mathbf1_{\{|f(\bm x)|>L_{N,\delta}\}}\,\rd\bm x\\
	&\le
	L_{N,\delta}^{1-p}\|f\|_{L^p}^{p}
	=
	\delta^{1-\frac{1}{p}}\|f\|_{L^p}N^{\frac{1}{p}-1}.
        \end{aligned}
    \end{equation}
	Moreover,
	\begin{equation}\label{eq:lp-truncation-l2}
	\|f_{N,\delta}\|_{L^2}^2
	=
	\int_{(0,1)^s}\min\{|f(\bm x)|,L_{N,\delta}\}^2\,\rd\bm x
	\le
	L_{N,\delta}^{2-p}\|f\|_{L^p}^{p}=\delta^{-\frac{2}{p}+1}\|f\|_{L^p}^2N^{\frac{2}{p}-1}.
	\end{equation}
	Applying Lemma~\ref{cor:dimension-independent-gain} to
	$f_{N,\delta}\in L^2$ and using
	\eqref{eq:lp-truncation-l2}, we obtain
	\begin{align*}
	\operatorname{Var}\left(
	\mathcal I_{m,\mathcal S}(f_{N,\delta})
	\right)
	\le
	\frac{C_\beta}{N}\|f_{N,\delta}\|_{L^2}^2\le
	C_\beta\delta^{-\frac{2}{p}+1}
	\|f\|_{L^p}^2N^{\frac{2}{p}-2}.
	\end{align*}
	Consequently, Chebyshev's inequality gives
	\begin{align}
	&\mathbb P\left\{
	\left|
	\mathcal I_{m,\mathcal S}(f_{N,\delta})-I(f_{N,\delta})
	\right|
	>
	\sqrt{C_\beta}\delta^{-\frac{1}{p}}
	\|f\|_{L^p}N^{\frac{1}{p}-1}
	\right\}
	\le\delta.
	\label{eq:lp-truncated-chebyshev}
	\end{align}
    
	Combining \eqref{eq:lp-truncation-sample-event},
	\eqref{eq:lp-truncation-bias}, and
	\eqref{eq:lp-truncated-chebyshev} yields
	\begin{align*}
	\mathbb P\left\{
	\left|\mathcal I_{m,\mathcal S}(f)-I(f)\right|
	>
	\left(\delta+\sqrt{C_\beta}\right)\delta^{-\frac{1}{p}}
	\|f\|_{L^p}N^{\frac{1}{p}-1}
	\right\}
	\le2\delta.
	\end{align*}
	The asserted median bound follows from
	Lemma~\ref{lem:median-amplification} with $\eta=2\delta$.
\end{proof}

\begin{remark}
  We note that a similar bound holds for the median of MC estimates; see \cite{kunsch2019optimal} for an example.
\end{remark}

\subsection{Error bounds for $\mathcal H_{\mathrm{wav},\alpha,s,p,q}$ integrands}\label{subsec:Haarintegrands}

Median RQMC can achieve a convergence rate faster than that of MC for integrands with regularity beyond $L^2$. Here we consider integrands in the Haar wavelet space $\mathcal  H_{\mathrm{wav},\alpha,s,p,q}$ (see Section~\ref{subsec:haar-space}). We assume throughout this section that $\alpha_p=\alpha-\left(1/p-1/2\right)_+>0$, which implies that $\mathcal H_{\mathrm{wav},\alpha,s,p,q}\subseteq \mathcal H_{\mathrm{wav},\alpha_p,s,2,q}\subseteq L^2$ (see Lemma~\ref{lem:HinL2}).

 To facilitate the analysis, we introduce the ANOVA decomposition \cite{meandim}:
    $$f(\bm x)=\sum_{u\subseteq 1{:}s} f_u(\bsx_u),$$
    where each component $f_u:(0,1)^{|u|}\to \real$ is defined recursively by $f_\emptyset=I(f)$ and
    \begin{equation}\label{eqn:fudef}
     f_u(\bm x_u)
	:=
	\int_{(0,1)^{s-|u|}}
	f(\bm x_u,\bm x_{-u})\,\rd\bm x_{-u}-\sum_{\substack{v\subsetneqq u}}f_v(\bsx_v).   
    \end{equation}
Let $\supp(\bsh)=\{j\in 1{:}s\mid h_j\neq 0\}$ denote the support of $\bsh=(h_1,\dots,h_s)$. As shown in \cite[Section 3]{pan2026dimension}, the Walsh coefficients $\widehat f(\bm{h})=\widehat {f_u}(\bsh_u)$ if $\supp(\bsh)=u$, and $\widehat {f_u}(\bsh_u)=0$ if $\supp(\bsh_u)\neq u$. Then by Lemma~\ref{lem:HaarWalshBlock}, the Haar detail projection $Q_{\bm j}f=Q_{\bm j_u}f_u$ if $\supp(\bm j)=u$, and $Q_{\bm j_u}f_u=0$ if $\supp(\bm j_u)\neq u$. Plugging these estimates into \eqref{eq:haar-wavelet-norm} yields
\begin{multline}\label{eqn:fuleqf}
 \|f_u\|^q_{\mathrm{wav},\alpha,|u|,p,q}=\sum_{\bm j_u\in\mathbb N_0^{|u|}}
		b^{q(\alpha-1/p+1/2)|\bm j_u|_1}
		\left(
		\sum_{\bm k\in\Delta_{\bm j_u-\bm 1}}
		\sum_{\bm i\in\nabla_{\bm j_u}}
		\left|
		\left\langle f_u,\psi_{\bm i,\bm k}^{\bm j}\right\rangle
		\right|^p
		\right)^{q/p}\\ 
        =  \sum_{\substack{\bm j\in\mathbb N_0^{s}\\ \supp(\bm j)=u }}
		b^{q(\alpha-1/p+1/2)|\bm j|_1}
		\left(
		\sum_{\bm k\in\Delta_{\bm j-\bm 1}}
		\sum_{\bm i\in\nabla_{\bm j}}
		\left|
		\left\langle f,\psi_{\bm i,\bm k}^{\bm j}\right\rangle
		\right|^p
		\right)^{q/p}\leq \|f\|^q_{\mathrm{wav},\alpha,s,p,q}.    
\end{multline}
 In particular, each component $f_u\in \mathcal H_{\mathrm{wav},\alpha,|u|,p,q}$.

Next, we introduce the frequency set used in the error analysis. For nonempty $u\subseteq1{:}s$ and $T\in\mathbb R$, let
\begin{equation*}
	K_{u}(T):=\{\bm h\in\mathbb N_0^s: \supp(\bsh)=u, \mu_1(\bm h)\le T\},
\end{equation*}
where $\mu_1(\bm h)$ is given by \eqref{eqn:bsmu}. The following lemma bounds the cardinality of $K_{u}(T)$.
\begin{lemma}\label{lem:Ku-size}
	Let $u\subseteq 1{:}s$ be nonempty, $T\in\mathbb R$, and $T_+=\max(T,0)$. Then
	\begin{equation*}
		|K_{u}(T)|
		\le
        b^{T}
		\left(\frac{b-1}{b}\right)^{|u|-1}
		\frac{(T_+)^{|u|-1}}{(|u|-1)!}.
	\end{equation*}
\end{lemma}
\begin{proof}
	If $T<|u|$, the set is empty and the inequality is trivial. Assume henceforth that $T\ge |u|$. Since the number of integers $h\in\mathbb N$ with $\mu_1(h)=\ell$ equals $(b-1)b^{\ell-1}$, partitioning $K_{u}(T)$ by the value of $\Vec{\bm{\mu}}_1(\bm{h})$ gives
    $$|K_{u}(T)|=\sum_{n=|u|}^{\lfloor T\rfloor}\sum_{\substack{\bsl_u\in \natu^{|u|}\\ |\bsl_u|_1=n}}\left(\prod_{j\in u} (b-1)b^{\ell_j-1}\right)=
		\left(\frac{b-1}{b}\right)^{|u|}
		\sum_{n=|u|}^{\lfloor T\rfloor}
		\binom{n-1}{|u|-1}b^{n}. $$
    The proof is complete after bounding
    \begin{equation}\label{eqn:cardbound}
        \sum_{n=|u|}^{\lfloor T\rfloor}
		\binom{n-1}{|u|-1}b^{n}\leq \frac{(\lfloor T\rfloor)^{|u|-1}}{(|u|-1)!}
		\sum_{n=0}^{\lfloor T\rfloor}b^n\leq \frac{(T_+)^{|u|-1}b^{T+1}}{(|u|-1)!(b-1)}.
    \end{equation}
\end{proof}

The next lemma bounds the sum of squared Walsh coefficients outside $K_{u}(T)$.

\begin{lemma}\label{lem:walsh-tail-outside-Ku}
	Let  $u\subseteq1{:}s$ be nonempty, $T\in\mathbb R$, and $T_+=\max(T,0)$. Then for $f\in\mathcal H_{\mathrm{wav},\alpha,s,p,q}$ with ANOVA components $f_u$,
    $$\sum_{\substack{\bm h\in\mathbb N_0^{s}\setminus K_{u}(T)\\\supp(\bsh)=u}}
	|\widehat f(\bm h)|^2\le C_{\alpha_p,q}^{|u|}\left(1+\frac{ (T_+)^{|u|-1}}{(|u|-1)!}\right)^{(1-2/q)_+} b^{-2\alpha_p T}
	\|f_u\|^2_{\mathrm{wav},\alpha,|u|,p,q},$$
    where $C_{\alpha_p,q}=1$ if $q\in [1,2]$ and 
    $$C_{\alpha_p,q}=\left(1-b^{-\frac{2\alpha_pq}{q-2}}\right)^{-1+2/q} \quad \text{if } q>2.$$
\end{lemma}

\begin{proof}
Since $\widehat f(\bm h)=\widehat {f_u}(\bm h_u)$ when $\supp(\bsh)=u$, the left hand side 
\begin{equation*}
   \sum_{\substack{\bm h\in\mathbb N_0^{s}\setminus K_{u}(T)\\\supp(\bsh)=u}}
	|\widehat f(\bm h)|^2=\sum_{\substack{\bm h_u\in\mathbb N^{|u|}\\ \mu_1(\bsh_u)\geq T }}
	|\widehat {f_u}(\bm h_u)|^2=\sum_{\substack{\bsl_u \in \mathbb N^{|u|}\\|\bm \ell_u|_1\geq T}}\|Q_{\bm \ell_u}f_u\|_{L^2}^2, 
\end{equation*}
where the second equality follows from Lemma~\ref{lem:HaarWalshBlock}.   

When $q\in [1,2]$, applying \eqref{Walsh_bound} to $f_u\in \mathcal H_{\mathrm{wav},\alpha,|u|,p,q}$ yields
\begin{equation}\label{eqn:qsmaller2}
   \sum_{\substack{\bsl_u \in \mathbb N^{|u|}\\|\bm \ell_u|_1\geq T}}\|Q_{\bm \ell_u}f_u\|_{L^2}^2
   \leq \frac{1}{b^{2 \alpha_p T}} \sum_{\substack{\bsl_u \in \mathbb N^{|u|}}}b^{2\alpha_p |\bm \ell_u|_1}\|Q_{\bm \ell_u}f_u\|_{L^2}^2\leq \frac{\|f_u\|^2_{\mathrm{wav},\alpha,|u|,p,2}}{b^{2\alpha_p T}}  .
\end{equation}
The conclusion then follows from $\|f_u\|_{\mathrm{wav},\alpha,|u|,p,2}\leq \|f_u\|_{\mathrm{wav},\alpha,|u|,p,q}$ (see \eqref{eqn:Hnormpq}). 

  When $q>2$, let $q':=q/(q-2)$ so that $1-2/q=1/q'$. Set $T_*=\max(|u|,\lceil T\rceil)$. H\"older's inequality and \eqref{Walsh_bound} give
  \begin{equation}\label{eqn:LqLq'}
      \begin{aligned}
       \sum_{\substack{\bm{\ell}_u \in \mathbb N^{|u|}\\|\bm{\ell}_u|_1\geq T}}\|Q_{\bm{\ell}_u}f_u\|_{L^2}^2
    &\leq 
    \left(\sum_{\substack{\bm{\ell}_u \in \mathbb N^{|u|}\\|\bm{\ell}_u|_1\geq T}}
    b^{q\alpha_p|\bm{\ell}_u|_1}\|Q_{\bm{\ell}_u}f_u\|_{L^2}^q\right)^{2/q}
    \left(\sum_{\substack{\bm{\ell}_u \in \mathbb N^{|u|}\\|\bm{\ell}_u|_1\geq T}}
    b^{-2q' \alpha_p |\bm{\ell}_u|_1}\right)^{1/q'} \\
    &\leq 
    \|f_u\|^2_{\mathrm{wav},\alpha,|u|,p,q}
    \left(\sum_{n= T_*}^\infty \binom{n-1}{|u|-1}
		b^{-2q'\alpha_p n}\right)^{1/q'}.    
      \end{aligned}
  \end{equation}
  By \cite[Lemma 13.24]{dick:pill:2010},
  $$\sum_{n= T_*}^\infty \binom{n-1}{|u|-1}
		b^{-2q'\alpha_p n}\leq b^{-2q'\alpha_p T_*}\binom{T_*-1}{|u|-1}\left(1-b^{-2q'\alpha_p}\right)^{-|u|}.$$   
        Further bounding $\binom{T_*-1}{|u|-1}\leq 1+\frac{ (T_+)^{|u|-1}}{(|u|-1)!}$  yields the claimed upper bound.
\end{proof}

Finally, to apply Lemma~\ref{thm:walsh-error-scrambled-net}, we need the following upper bound on the probability that $\bm h\in D_{\bm{C},m}^\perp(\bm M)$ under the coarse scrambling.

\begin{lemma}\label{lem:dualnetbound}
If $\mathcal S$ is a  $(\log n)^{1+\beta}$-reduced GNS with $\beta>0$, then for every nonzero $\bm h\in \mathbb N_0^{s}$ and the random dual net $D_{\bm{C},m}^\perp(\bm M)$ defined in Lemma~\ref{thm:walsh-error-scrambled-net},
	\begin{equation*}
	\mathbb P\{\bm h\in D_{\bm C,m}^{\perp}(\bm M)\}
	\le
	C_\beta b^{-m},
	\end{equation*}
    where $C_\beta$ is given by Lemma~\ref{cor:dimension-independent-gain}.
\end{lemma}
\begin{proof}
    Consider the $L^2$ integrand $f_{\bsh}(\bsx)=\operatorname{wal}_{\bm h}(\bm x)$ and the corresponding RQMC estimator ${\mathcal I}_{m,\mathcal S}(f_{\bsh})$. By Lemma~\ref{thm:walsh-error-scrambled-net},
    \begin{equation*}
	\mathbb E_{\bm D}
	\left[
	\left|
	{\mathcal I}_{m,\mathcal S}(f_{\bsh})
	-
	I(f_{\bsh})
	\right|^2	
	\right]
	=
	\sum_{\bm0\neq\bm h'\in D_{\bm{C},m}^\perp(\bm M)}
	|\widehat f_{\bsh}(\bm h')|^2=\bsone\{\bsh\in D_{\bm{C},m}^\perp(\bm M)\}.
	\end{equation*}
    Taking expectation with respect to $\bm M$ yields $$\operatorname{Var}\left({\mathcal I}_{m,\mathcal S}(f_{\bsh})\right)=\mathbb P\{\bm h\in D_{\bm C,m}^{\perp}(\bm M)\}.$$
    Meanwhile, Lemma~\ref{cor:dimension-independent-gain} implies the bound
    $$\operatorname{Var}\left({\mathcal I}_{m,\mathcal S}(f_{\bsh})\right)\le
	 \frac{C_{\beta}}{b^m}\sigma_{MC}^2(f_{\bsh})=C_{\beta}b^{-m},$$
     which completes the proof.
\end{proof}

\begin{theorem}\label{thm:median-weighted-probability}
For nonzero $f\in\mathcal H_{\mathrm{wav},\alpha,s,p,q}$ with ANOVA components $f_u$, let
$$\gamma_{u,\alpha,p,q}:=\frac{\|f_u\|_{\mathrm{wav},\alpha,|u|,p,q}}{\|f\|_{\mathrm{wav},\alpha,s,p,q}}.$$
If $\mathcal S$ is a  $(\log n)^{1+\beta}$-reduced GNS with $\beta>0$, then for $\delta\in (0,1/8)$ and odd $R\in\natu$,
	\begin{equation}\label{eqn:Haarmedianbound}
	\mathbb P\left\{
	\left|{\mathcal M}_{m,R,\mathcal S}(f)-I(f)\right|
	> 
	\frac{\Psi_{\alpha,s,p,q,\gamma,\beta}(m) }{(b^{m}\delta)^{\alpha_p+1/2}} \|f\|_{\mathrm{wav},\alpha,s,p,q}
	\right\}
	\le
	\frac{1}{2}(8\delta)^{\frac{R+1}{2}},
	\end{equation}
    where, with $C_\beta$ as in Lemma~\ref{cor:dimension-independent-gain} and $C_{\alpha_p,q}$ as in Lemma~\ref{lem:walsh-tail-outside-Ku},
\begin{align}
\Psi_{\alpha,s,p,q,\gamma,\beta}(m)
&=
\left(1+C_\beta\psi_{\alpha,s,p,q,\gamma}(m)\right)^{\alpha_p}
\left(C_\beta \psi'_{\alpha,s,p,q,\gamma}(m)\right)^{1/2}, \notag\\
\psi_{\alpha,s,p,q,\gamma}(m) 
&=
\sum_{\emptyset\ne u\subseteq 1{:}s}
\left(\frac{b-1}{b}\right)^{|u|-1}
\frac{m^{|u|-1}}{(|u|-1)!}
(\gamma_{u,\alpha,p,q})^{\frac{1}{\alpha_{p}+1/2}}, \label{eqn:psim}\\
\psi'_{\alpha,s,p,q,\gamma}(m)
&=
\sum_{\emptyset\ne u\subseteq 1{:}s}
C_{\alpha_p,q}^{|u|}
\left(1+\frac{m^{|u|-1}}{(|u|-1)!}\right)^{(1-2/q)_+}
(\gamma_{u,\alpha,p,q})^{\frac{1}{\alpha_{p}+1/2}}. \label{eqn:psim2}
\end{align}
\end{theorem}

\begin{proof} Denote $\mathcal U=\{\emptyset\neq u\subseteq 1{:}s : \gamma_{u,\alpha,p,q}>0\}$. For every $u\in \mathcal U$, let
\begin{equation}\label{eqn:Tu}
T_{u}
	=	
	m+\log_b\delta-\log_b(1+C_\beta\psi_{\alpha,p,q,\gamma}(m))
	+\frac{\log_b(\gamma_{u,\alpha,p,q})}{\alpha_{p}+1/2},
\end{equation}
and define the event
	\begin{equation*}
	G(\bm M)
	=
	\left\{
	K_{u}(T_u)\cap D_{\bm C,m}^{\perp}(M)=\emptyset \ \forall u\in \mathcal{U}
	\right\}.
	\end{equation*}
   Applying a union bound to the complement of $G(\bm M)$, we obtain
   \begin{equation}\label{eqn:unionbound}
       \begin{aligned}
           \mathbb P(G(\bm M)^c)
     &\overset{(1)}{\leq} \frac{C_\beta}{b^m} \sum_{u\in \mathcal{U}} b^{T_u}\left(\frac{b-1}{b}\right)^{|u|-1}
		\frac{((T_u)_+)^{|u|-1}}{(|u|-1)!}
		 \\
        &\overset{(2)}{\leq} \frac{\delta C_\beta}{1+C_\beta\psi_{\alpha,s,p,q,\gamma}(m)}  \sum_{u\in \mathcal{U}} \left(\frac{b-1}{b}\right)^{|u|-1}
		\frac{m^{|u|-1}}{(|u|-1)!}(\gamma_{u,\alpha,p,q})^{\frac{1}{\alpha_p+1/2}},
       \end{aligned}
   \end{equation}
    where (1) uses Lemmas~\ref{lem:Ku-size} and \ref{lem:dualnetbound}, and (2) follows by substituting \eqref{eqn:Tu} and bounding $(T_u)_+\leq m$. Substituting \eqref{eqn:psim} then implies $\mathbb P(G(\bm M)^c)\leq \delta$.
    
	 Since $G(\bm M)$ is independent of the random shifts $\bm D$, conditioning on the random matrices $\bm M$ and applying Lemma~\ref{thm:walsh-error-scrambled-net} yields
\begin{align*}
    &\mathbb E\left[
	\left|{\mathcal I}_{m,\mathcal S}(f)-I(f)\right|^2\mathbf1_{G(\bm M)}
	\right]=
	\mathbb E\left[
	\sum_{\bm0\ne\bm h\in D_{\bm C,m}^{\perp}(M)}
	|\widehat f(\bm h)|^2\mathbf1_{G(\bm M)}
	\right]\\
	\overset{(3)}{\leq}&
	\mathbb E\left[
	\sum_{u\in \mathcal U}
	\sum_{\substack{\bm h\in\mathbb N_0^{s}\setminus K_{u}(T_u)\\\supp(\bsh)=u}}
	|\widehat f(\bm h)|^2
	\bsone\{\bsh\in D_{\bm{C},m}^\perp(\bm M)\}
	\right]
	\overset{(4)}{\leq}
    \frac{C_{\beta}}{b^{m}}
	\sum_{u\in \mathcal U}\sum_{\substack{\bm h\in\mathbb N_0^{s}\setminus K_{u}(T_u)\\\supp(\bsh)=u}}
	|\widehat f(\bm h)|^2,
\end{align*}
where (3) uses the fact that $\widehat f(\bsh)=0$  for $\bsh\neq\bm0$ and $\supp(\bsh)\notin\mathcal U$, and that $K_{u}(T_u)\cap D_{\bm{C},m}^\perp(\bm M)=\emptyset$ under $G(\bm M)$; (4) follows from Lemma~\ref{lem:dualnetbound}.

Applying Lemma~\ref{lem:walsh-tail-outside-Ku} then gives
     
	\begin{align*}
	&\mathbb E\left[
	\left|{\mathcal I}_{m,\mathcal S}(f)-I(f)\right|^2\mathbf1_{G(\bm M)}
	\right]\\
   \leq &\frac{C_{\beta}}{b^{m}}\sum_{u\in\mathcal U} C_{\alpha_p,q}^{|u|}\left(1+\frac{ ((T_u)_+)^{|u|-1}}{(|u|-1)!}\right)^{(1-2/q)_+} b^{-2\alpha_p T_u}
	\|f_u\|^2_{\mathrm{wav},\alpha_p,|u|,p,q}\\
    \overset{(5)}{\leq}& \frac{C_\beta\psi_{\alpha,s,p,q,\gamma}'(m)\left(1+C_\beta{\psi_{\alpha,s,p,q,\gamma}(m)}\right)^{2\alpha_p}}{b^{(2\alpha_p+1) m}\delta^{2\alpha_p}}\|f\|^2_{\mathrm{wav},\alpha,s,p,q},
	\end{align*}
	where (5) follows upon bounding $(T_u)_+\leq m$ and substituting \eqref{eqn:psim2} and \eqref{eqn:Tu}.
    
	Finally, by Markov's inequality and $\mathbb P(G(\bm M)^c)\le\delta$,
    \begin{align*}
       & \mathbb P\left\{
	\left|{\mathcal I}_{m,\mathcal S}(f)-I(f)\right|
	>
	\frac{\Psi_{\alpha,s,p,q,\gamma,\beta}(m) }{(b^{m}\delta)^{\alpha_p+1/2}} \|f\|_{\mathrm{wav},\alpha,s,p,q}
	\right\}\\
	\leq &  \mathbb P(G(\bm M)^c)
	+
	\left({b^{m}\delta}\right)^{1+2\alpha_p} \frac{\mathbb E\left[
	\left|{\mathcal I}_{m,\mathcal S}(f)-I(f)\right|^2\mathbf1_{G(\bm M)}
	\right]}{\Psi_{\alpha,s,p,q,\gamma,\beta}(m)^2\|f\|^2_{\mathrm{wav},\alpha,s,p,q}}
	\le
	2\delta .  
    \end{align*}
	The asserted bound follows from
	Lemma~\ref{lem:median-amplification} with $\eta=2\delta$.
\end{proof}

Since $\Psi_{\alpha,s,p,q,\gamma,\beta}(m)=O(m^{(\alpha_p+(1/2-1/q)_+)(s-1)})$ as $m\to\infty$, we have, with failure probability decaying exponentially in $R$,
\[
\left|{\mathcal M}_{m,R,\mathcal S}(f)-I(f)\right|
=
O\!\left(m^{\{\alpha_p+(1/2-1/q)_+\}(s-1)}b^{-(\alpha_p+1/2)m}\right).
\]
 Under additional assumptions on the relative norm $\gamma_{u,\alpha,p,q}$, the next corollary gives a dimension-independent bound for this error.

 \begin{corollary}\label{cor:tractability}
     Under the assumptions of Theorem~\ref{thm:median-weighted-probability}, assume there exists a nonnegative sequence $\{\Upsilon_j\}_{j\ge1}$ satisfying
     \begin{equation}\label{eqn:tractability}
       \sum_{j=1}^\infty \Upsilon^{\frac{1}{\alpha_p+1/2}}_j<\infty \quad \text{and} \quad  \gamma_{u,\alpha,p,q}\le (|u|!)^{\alpha_p+1/2}\prod_{j\in u}\Upsilon_j \ \forall\emptyset\ne u\subseteq1{:}s. 
     \end{equation}
     Then, for every $\varepsilon>0$, there is a constant $C_{\varepsilon,\Upsilon,\alpha,p,q,\beta}$ independent of $s$ and $m$ such that
    \begin{equation*}
	\mathbb P\left\{
	\left|{\mathcal M}_{m,R,\mathcal S}(f)-I(f)\right|
	> 
	\frac{C_{\varepsilon,\Upsilon,\alpha,p,q,\beta}}{b^{(\alpha_p+1/2-\varepsilon)m}\delta^{\alpha_p+1/2}} \|f\|_{\mathrm{wav},\alpha,s,p,q}
	\right\}
	\le
	\frac{1}{2}(8\delta)^{\frac{R+1}{2}}.
	\end{equation*}
 \end{corollary}

\begin{proof}
	It suffices to show that, for every $\varepsilon>0$, $\psi_{\alpha,s,p,q,\gamma}(m)$ and $\psi'_{\alpha,s,p,q,\gamma}(m)$ are bounded by $C_{\varepsilon,\Upsilon,\alpha,p,q} b^{\varepsilon m}$ with $C_{\varepsilon,\Upsilon,\alpha,p,q}$ independent of $s$ and $m$. 
    
    Substituting \eqref{eqn:tractability} into \eqref{eqn:psim} yields
    \begin{equation}\label{eqn:tractabilitysum}
        \begin{aligned}
             \psi_{\alpha,s,p,q,\gamma}(m)\leq &\sum_{\emptyset\ne u\subseteq1{:}s}
	\left(\frac{b-1}{b}\right)^{|u|-1}
		\frac{m^{|u|-1}}{(|u|-1)!} 
	|u|!\prod_{j\in u}\Upsilon^{\frac{1}{\alpha_p+1/2}}_j\\
    \overset{(1)}{\leq}& \frac{1}{m} \sum_{u\subseteq 1{:}s}
	\prod_{j\in u}\left(2m\Upsilon^{\frac{1}{\alpha_p+1/2}}_j\right)\leq \frac{1}{m} \prod_{j=1}^\infty \left(1+2m\Upsilon^{\frac{1}{\alpha_p+1/2}}_j\right),
        \end{aligned}
    \end{equation}
where (1) uses $|u|\leq 2^{|u|}$. Since $\Upsilon^{\frac{1}{\alpha_p+1/2}}_j$ is summable, \cite[Lemma 3]{HICKERNELL2003286} proves the claim.

Next, substituting \eqref{eqn:tractability} into \eqref{eqn:psim2} yields
\begin{align*}
    \psi_{\alpha,s,p,q,\gamma}'(m)\leq &\sum_{\emptyset\ne u\subseteq1{:}s}C_{\alpha_p,q}^{|u|}\left(1+\frac{ m^{|u|-1}}{(|u|-1)!}\right)^{(1-2/q)_+} |u|!\prod_{j\in u}\Upsilon^{\frac{1}{\alpha_p+1/2}}_j\\
    \overset{(2)}{\leq} &\frac{1}{m} \prod_{j=1}^\infty \left(1+2C_{\alpha_p,q}m\Upsilon^{\frac{1}{\alpha_p+1/2}}_j\right)+2\sum_{\substack{\emptyset\neq u\subseteq \natu\\|u|<\infty}} C_{\alpha_p,q}^{|u|} |u|!\prod_{j\in u}\Upsilon^{\frac{1}{\alpha_p+1/2}}_j,
\end{align*}
where (2) uses $|a|^\lambda\leq 1+|a|$ for $\lambda=(1-2/q)_+\in [0,1]$. The desired bound follows by applying \cite[Lemma 3]{HICKERNELL2003286} to the first term and \cite[Theorem 2]{Pan2026Sharp} to the second term.
\end{proof}

\subsection{Error bounds for integrands with $\mathcal H_{\mathrm{wav},\alpha,s,p,q}$ derivatives}\label{subsec:Haarderiv}

The Haar wavelet space $\mathcal H_{\mathrm{wav},\alpha,s,p,q}$ is well suited for characterizing fractional smoothness \cite{gnewuch2026qmc}. To treat weakly differentiable integrands, we define, for $r\ge 1$,
$$\mathcal H_{\mathrm{mix},r,\alpha,s,p,q}=\left\{f\in W_{\mathrm{mix}}^{r,1}:f^{(\bm\tau)}\in \mathcal H_{\mathrm{wav},\alpha,s,p,q}\ \forall \bm\tau\in \{0{:}r\}^s\right\}.$$
We equip the space with the norm
$$\|f\|_{\mathrm{mix},r,\alpha,s,p,q}=
	\left(\sum_{\bm\tau\in\{0{:}r\}^s}
	\left\|
	f^{(\bm\tau)}
	\right\|^q_{\mathrm{wav},\alpha,s,p,q}\right)^{1/q},$$
We assume two cases: either $\alpha=0$ and $p=q=2$, in which case $\mathcal H_{\mathrm{wav},\alpha,s,p,q}=L^2$ and the space coincides with $W_{\mathrm{mix}}^{r,2}$; or $\alpha_p=\alpha-\left(1/p-1/2\right)_+>0$, in which case the space embeds into $W_{\mathrm{mix}}^{r,2}$ (see Lemma 2.6).

Our analysis follows the framework in Section~\ref{subsec:Haarintegrands}. Let $f_u$ be the ANOVA component of $f$ corresponding to the subset $u$. By \eqref{eqn:fudef}, for every $\bm\tau\in\{0{:}r\}^s$ with $\supp(\bm\tau)=u$, the weak derivative $f_u^{(\bm\tau_u)}$ satisfies
\begin{equation*}
f_u^{(\bm\tau_u)}(\bm x_u)
=
\int_{(0,1)^{s-|u|}}
f^{(\bm\tau)}(\bm x_u,\bm x_{-u})\,d\bm x_{-u}.
\end{equation*}
Hence,
\[
\widehat{f_u^{\,(\bm\tau_u)}}(\bsh_u)=\widehat{f^{\,(\bm\tau)}}(\bsh)
\quad\text{for } \bsh_u\in\mathbb N_0^{|u|},\ \bsh=(\bsh_u,\bm 0_{-u}).
\]
Lemma~\ref{lem:HaarWalshBlock} then implies that $Q_{\bm j_u}f^{(\bm\tau_u)}_u=Q_{\bm j}f^{(\bm\tau)}$ if $\bm \tau_u\in \{1{:}r\}^{|u|}$, $\bm \tau=(\bm \tau_u,\bm 0_{-u})$, and $\bm j=(\bm j_u,\bm 0_{-u})$. An argument similar to \eqref{eqn:fuleqf} shows
\[
\|f_u\|_{\mathrm{mix}^*,r,\alpha,|u|,p,q}
:=
\left(
\sum_{\bm\tau_u\in\{1{:}r\}^{|u|}}
\left\|
f_u^{(\bm\tau_u)}
\right\|_{\mathrm{wav},\alpha,|u|,p,q}^q
\right)^{1/q}
\le
\|f\|_{\mathrm{mix},r,\alpha,s,p,q}.
\]
Although $\|\cdot\|_{\mathrm{mix}^*,r,\alpha,|u|,p,q}$ defines only a semi-norm in general, the ANOVA component $f_u$ satisfies the additional zero-mean condition
\[
\int_0^1 f_u(\bsx_u)\,dx_j=0 \quad\text{for every } j\in u,
\]
so $\|f_u\|_{\mathrm{mix}^*,r,\alpha,|u|,p,q}=0$ implies $f_u=0$ by the Poincaré inequality (see \cite{effdimsobononper}).

We next introduce the frequency sets used in the error analysis, analogous to $K_u(T)$ from Section~\ref{subsec:Haarintegrands}. For nonempty \(u\subseteq 1{:}s\) and \(T\in\mathbb R\), define
\begin{equation*}
K_{u,r,\alpha_p}(T):=\{\bm h\in\mathbb N_0^{s}:\supp(\bsh)=u,\;
(1-\alpha_p)\mu_r(\bm h)+\alpha_p\mu_{r+1}(\bm h)\le T\},
\end{equation*}
where $\mu_r(\bm h)$ and $\mu_{r+1}(\bm h)$ are as in \eqref{eqn:bsmu}.
The following lemma bounds the cardinality of \( K_{u,r,\alpha_p}(T) \). Its proof is given in Appendix~\ref{app:lem:Ku-size-deriv}.

\begin{lemma}\label{lem:Ku-size-deriv}
For nonempty \(u\subseteq 1{:}s\), \(r\in\mathbb N\), $\alpha_p\geq 0$, \(T\in\mathbb R\), and \(T_+:=\max(T,0)\), there is a constant $C_{r,\alpha_p}$ depending on $r$ and $\alpha_p$ such that
 $$|K_{u,r,\alpha_p}(T)|\leq C^{|u|}_{r,\alpha_p} b^{T/(r+\alpha_p)} \frac{ (T_+)^{|u|-1}}{(|u|-1)!}.$$
\end{lemma}

The next lemma bounds the sum of squared Walsh coefficients outside \(K_{u,r,\alpha_p}(T)\). Its proof is technical and deferred to Appendix~\ref{app:liu-block-proof}.

\begin{lemma}\label{lem:walsh-tail-outside-Ku-deriv}
Let $r\in\natu$, $u\subseteq1{:}s$ be nonempty, $T\in\mathbb R$, and \(T_+:=\max(T,0)\). 
Then for $f\in\mathcal H_{\mathrm{mix},r,\alpha,s,p,q}$ with ANOVA components $f_u$, 
\begin{align*}
\sum_{\substack{\bm h\in\mathbb N_0^{s}\setminus K_{u,r,\alpha_p}(T)\\\supp(\bsh)=u}}
	|\widehat f(\bm h)|^2
\le
C_{r,\alpha_p,q}^{|u|}
\mathfrak P_{|u|,r,q}(T)b^{-2T}
\|f_u\|_{\mathrm{mix}^*,r,\alpha,|u|,p,q},
\end{align*}
where \(C_{r,\alpha_p,q}\) is a constant depending on $r,\alpha_p$, and $q$, and 
\begin{equation*}
 \mathfrak P_{|u|,r,q}(T):=
\mathcal E_{r|u|}(T_+)
\left(\mathcal E_{|u|-1}(T_+)\right)^{(1-2/q)_+},\quad  \mathcal  E_n(x)=\sum_{\nu=0}^n\frac{x^\nu}{\nu!}.   
\end{equation*}
\end{lemma}

\begin{theorem}\label{thm:median-weighted-probability2}
For nonzero $f\in\mathcal H_{\mathrm{mix},r,\alpha,s,p,q}$ with ANOVA components $f_u$, define
\[
\gamma_{u,r,\alpha,p,q}
:=
\frac{\|f_u\|_{\mathrm{mix}^*,r,\alpha,|u|,p,q}}
{\|f\|_{\mathrm{mix},r,\alpha,s,p,q}}.
\]
If $\mathcal S$ is a $(\log n)^{1+\beta}$-reduced GNS with $\beta>0$,
then for $\delta\in(0,1/8)$ and odd $R\in\natu$,
\begin{equation*}
\mathbb P\left\{
\left|{\mathcal M}_{m,R,\mathcal S}(f)-I(f)\right|
>
\frac{\Psi_{r,\alpha,s,p,q,\gamma,\beta}(m)}
{(b^m\delta)^{r+\alpha_p+1/2}}
\|f\|_{\mathrm{mix},r,\alpha,s,p,q}
\right\}
\le
\frac12(8\delta)^{\frac{R+1}{2}},
\end{equation*}
where, with $C_\beta$ as in Lemma~\ref{cor:dimension-independent-gain}, with $C_{r,\alpha_p}$ as in Lemma~\ref{lem:Ku-size-deriv}, and with $C_{r,\alpha_p,q}$ and $\mathfrak P_{|u|,r,q}$ as in Lemma~\ref{lem:walsh-tail-outside-Ku-deriv},
\begin{align*}
 \Psi_{r,\alpha,s,p,q,\gamma,\beta}(m)
 &=
\left(1+C_\beta\psi_{r,\alpha,s,p,q,\gamma}(m)\right)^{r+\alpha_p}
\left(C_\beta\psi'_{r,\alpha,s,p,q,\gamma}(m)\right)^{1/2}, \\
\psi_{r,\alpha,s,p,q,\gamma}(m)
&=
\sum_{\emptyset\ne u\subseteq1{:}s}
C_{r,\alpha_p}^{|u|}
\frac{((r+\alpha_p)m)^{|u|-1}}{(|u|-1)!}
(\gamma_{u,r,\alpha,p,q})^{\frac{1}{r+\alpha_p+1/2}},\\
\psi'_{r,\alpha,s,p,q,\gamma}(m)
&=
\sum_{\emptyset\ne u\subseteq1{:}s}
C_{r,\alpha_p,q}^{|u|}
\mathfrak P_{|u|,r,q}((r+\alpha_p)m)
(\gamma_{u,r,\alpha,p,q})^{\frac{1}{r+\alpha_p+1/2}}.
\end{align*}
\end{theorem}

\begin{proof}
The proof parallels that of Theorem~\ref{thm:median-weighted-probability}, with the following modifications: we set, 
for every $u\in \mathcal U= \{\emptyset\neq u\subseteq 1{:}s : \gamma_{u,r,\alpha,p,q}>0\}$,
\[
T_u
=
(r+\alpha_p)\left(
m+\log_b\delta
-\log_b(1+C_\beta\psi_{r,\alpha,s,p,q,\gamma}(m))
+\frac{\log_b\gamma_{u,r,\alpha,p,q}}
{r+\alpha_p+1/2}
\right),
\]
and define
\[
G(\bm M)
=
\left\{
K_{u,r,\alpha_p}(T_u)\cap D_{\bm C,m}^{\perp}(\bm M)=\emptyset \ \forall
u\in \mathcal U
\right\}.
\]
The modified proof is then completed by bounding \((T_u)_+\le(r+\alpha_p)m\) and replacing Lemmas~\ref{lem:Ku-size} and~\ref{lem:walsh-tail-outside-Ku} with Lemmas~\ref{lem:Ku-size-deriv} and~\ref{lem:walsh-tail-outside-Ku-deriv}, respectively.
\end{proof}

Theorem~\ref{thm:median-weighted-probability2} shows that, with failure probability decaying
exponentially in $R$,
\[
|{\mathcal M}_{m,R,\mathcal S}(f)-I(f)|
=
O\!\left(
m^{\{r+\alpha_p+(1/2-1/q)_+\}(s-1)+rs/2}
b^{-(r+\alpha_p+1/2)m}
\right).
\]
Dimension-independent error bounds can be obtained under the following assumptions on $\gamma_{u,r,\alpha,p,q}$. The proof is given in Appendix~\ref{app:cor:deriv}.

\begin{corollary}\label{cor:deriv}
Under the assumptions of
Theorem~\ref{thm:median-weighted-probability2}, assume that there is a
nonnegative sequence $\{\Upsilon_j\}_{j\ge1}$ such that
\begin{equation}\label{eqn:tractability2}
\sum_{j=1}^\infty
\Upsilon_j^{\frac{1}{r+\alpha_p+1/2}}<\infty
\quad \text{and} \quad  
\gamma_{u,r,\alpha,p,q}
\le
(|u|!)^{r+\alpha_p+1/2}\prod_{j\in u}\Upsilon_j \ \forall\emptyset\ne u\subseteq1{:}
\end{equation}
for every nonempty $u\subseteq1{:}s$. Then, for every $\varepsilon>0$,
there is a constant $C_{\varepsilon,\Upsilon,r,\alpha,p,q,\beta}$,
independent of $s$ and $m$, such that
\begin{equation*}
\mathbb P\left\{
\left|{\mathcal M}_{m,R,\mathcal S}(f)-I(f)\right|
>
\frac{C_{\varepsilon,\Upsilon,r,\alpha,p,q,\beta}}
{b^{(r+\alpha_p+1/2-\varepsilon)m}
\delta^{r+\alpha_p+1/2}}
\|f\|_{\mathrm{mix},r,\alpha,s,p,q}
\right\}
\le
\frac12(8\delta)^{\frac{R+1}{2}}.
\end{equation*}
\end{corollary}

 \begin{remark}
 The tractability condition \eqref{eqn:tractability2} extends those in \cite[Remark 3]{pan2026dimension} and \cite[Remark 4]{goda2026quasi} by incorporating an order-dependent factor \((|u|!)^{r+\alpha_p+1/2}\) into the upper bound for \(\gamma_{u,r,\alpha,p,q}\). Bounds of this type arise naturally in the analysis of QMC methods for partial differential equations with random coefficients (see, e.g., \cite{kuo:nuye:2016}).
 \end{remark}

\section{Numerical results}\label{sec:numerics}

In this section, we numerically compare $(\log n)^2$-reduced Sobol' sequences with full Sobol' sequences under coarse scrambling, referring to them as reduced coarse scrambling (RCS) and full coarse scrambling (FCS), respectively. As benchmarks, we use linearly scrambled (LS) Sobol' sequences \cite{Matousek1998}, the Hankel random design (HRD) \cite{goda2026quasi}, and Monte Carlo (MC). The first four RQMC methods use the median of $R=15$ randomized replicates, each comprising $N=2^m$ points, whereas MC uses the mean of $RN$ independent uniform samples on $(0,1)^s$. Code is available in the
GitHub release\footnote{\url{https://github.com/ckzmath/median-coarse-scrambling-reproducibility/releases/tag/v1.0.0}}.

\subsection{Smooth integrands}

Following \cite{goda2026quasi,goda2024universal,pan2026automatic}, we consider the integrand
\begin{equation*}
  f_c(\bsx)
  =
  \prod_{j=1}^{128}
  \left[
    1+\frac{1}{j^{c+1}}
    \left(
      x_j^c-\frac{1}{1+c}
    \right)
  \right].
\end{equation*}

\begin{figure}[t]
  \centering
  \includegraphics[width=0.95\linewidth]
  {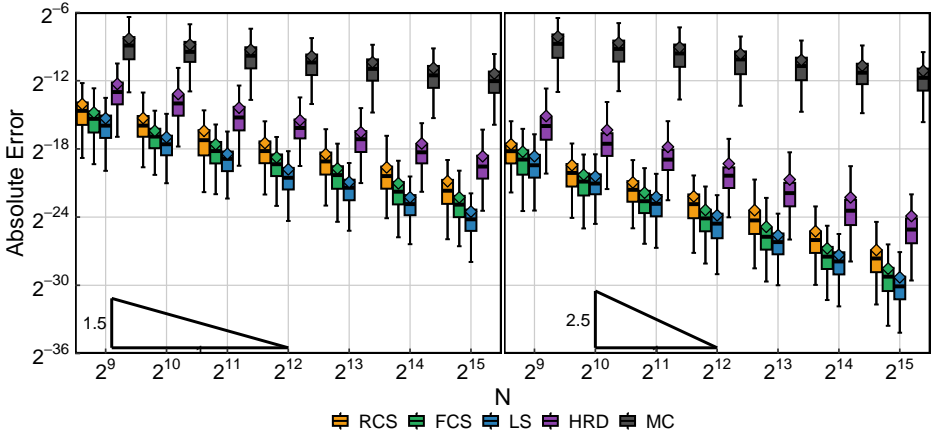}
  \caption{Boxplots of the absolute integration error for $f_c$ with $c=0.5$ (left) and $c=1.5$ (right). Each boxplot is based on $448$ independent replicates.}
  \label{fig:fc-boxplots}
\end{figure}

Figure~\ref{fig:fc-boxplots} compares the integration error for \(f_c\) with \(c=0.5\) and \(c=1.5\) across the five methods. As \(c\) increases, the four median RQMC methods automatically achieve faster convergence, owing to the improved smoothness and lower effective dimensionality; MC, by contrast, retains its characteristic \(N^{-1/2}\) rate.

Among the RQMC methods, LS performs the best, followed closely by FCS and RCS. This is consistent with the fact that the coordinate ordering matches their relative importance. Since the degree of the base polynomial is lower for earlier coordinates, the projection of Sobol' sequences onto these coordinates is equidistributed in a smaller base. As a result, the ANOVA components corresponding to the more important coordinates are integrated more accurately than those of lesser importance, which explains the observed practical advantage.

\begin{figure}[t]
  \centering
  \includegraphics[width=0.95\linewidth]
  {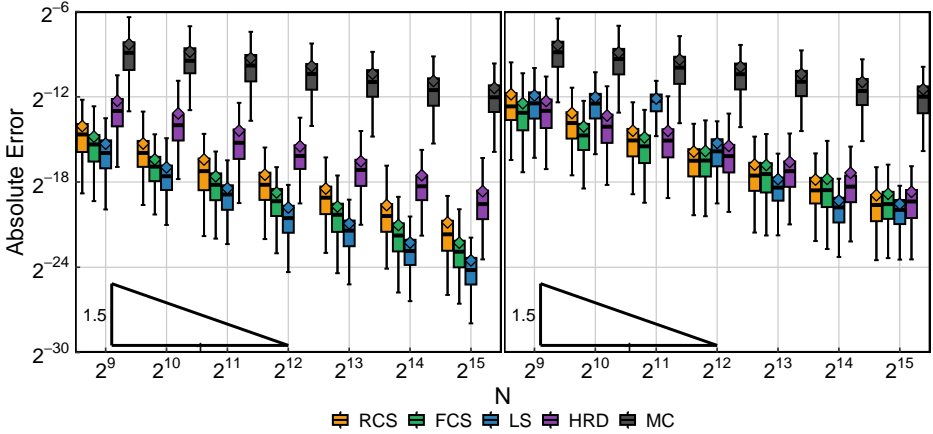}
\caption{Boxplots of the absolute integration error for $f_c$ with $c=0.5$ under
the original (left) and reversed (right) coordinate orderings. 
Each boxplot is based on $448$ independent replicates.}
  \label{fig:fc-ordering}
\end{figure}

Figure~\ref{fig:fc-ordering} reports the error for \(f_c\) with \(c=0.5\) under the original and reversed coordinate orderings, i.e., for \(f_c(x_1,\dots,x_{128})\) and \(f_c(x_{128},\dots,x_1)\). HRD and MC are unaffected by the reversal, as both methods are invariant under coordinate permutations. In contrast, the three Sobol'-based methods deteriorate, because the more important coordinates are now assigned larger-degree base polynomials; LS in particular degrades significantly and loses its leading position.

\subsection{Nonsmooth integrands} To examine the performance of our methods on nonsmooth integrands, we consider the following three examples:
\begin{align*}
	g_1(\bsx)
	&=
	\mathbf{1}
	\left\{
	\sum_{j=1}^{5}x_j^2\leq 1
	\right\},
	\\[1mm]
	g_2(\bm x)
	&=
	\prod_{j=1}^{128}
	\left[
	1+\frac{1}{j^2}
	\left(
	\sqrt{1-2M}
	\exp\!\left(M\Phi^{-1}(x_j)^2\right)-1
	\right)
	\right],
	\qquad M=0.3,
	\\[1mm]
	g_3(\bsx)
	&=
	\prod_{j=1}^{128}
	\left[
	1+\frac{1}{j^2}
	\left(
	\frac{1-\rho}{2^\rho}
	\left|x_j-\frac12\right|^{-\rho}-1
	\right)
	\right],
	\qquad \rho=0.25.
\end{align*}
Here, $\Phi$ denotes the standard normal distribution function.

The function \(g_1\) is a typical indicator function, frequently encountered in financial engineering applications \cite{he2014good,wang2013pricing}. Theorem~\ref{thm:l2-median-baseline} guarantees an RMSE rate of at least \(O(N^{-1/2})\). A refined analysis using the equidistribution property of Sobol' sequences yields the sharper rate \(O(N^{-(1+1/s)/2})=O(N^{-3/5})\) (see \cite{he2015conrate,liu2026randomized} for details).

The function \(g_2\) illustrates integrands with boundary singularities, which arise naturally in Gaussian integrals. With \(M=0.3\) and \(z_j=\Phi^{-1}(x_j)\) standard normal, the function \(\exp(Mz_j^2)\) belongs to \(L^p\) for \(p\in[1,5/3)\) but not for \(p=5/3\). Theorem~\ref{thm:lp-median-bound} then predicts a convergence rate of nearly \( O(N^{-2/5})\).

Finally, \(g_3\) is an \(L^2\) function with interior singularities. Using the base-\(2\) Haar wavelets \(\psi_{i,k}^{\ell}\) from \eqref{formula_Haar}, we expand \(|x-1/2|^{-\rho}\) as
\[
\begin{aligned}
|x-1/2|^{-\rho}
=
\frac{2^{9/4}}{3}
+
\sum_{\ell=1}^{\infty}\sum_{k=0}^{2^{\ell-1}-1}
c_{\ell,k}
\bigl(\psi_{0,k}^{\ell}(x)-\psi_{1,k}^{\ell}(x)\bigr),
\end{aligned}
\]
where \(c_{1,0}=0\), and for \(\ell\ge2\), with \(k_\ell=2^{\ell-2}\),
\[
c_{\ell,k}
=
\frac{2}{3}\,2^{-\ell/4}
\begin{cases}
(k_\ell-k)^{3/4}+(k_\ell-k-1)^{3/4}
-2\left(k_\ell-k-\frac12\right)^{3/4}, & k<k_\ell,\\[6pt]
2\left(k-k_\ell+\frac12\right)^{3/4}
-(k-k_\ell)^{3/4}
-(k-k_\ell+1)^{3/4}, & k\ge k_\ell.
\end{cases}
\]
Substituting the above expansion into the Haar wavelet norm \eqref{eq:haar-wavelet-norm} shows that \(g_3\in \mathcal H_{\mathrm{wav},\alpha,s,p,q}\) with \(\alpha\in(0,1/4)\) and \(p=q=2\). Theorem~\ref{thm:median-weighted-probability} then yields a convergence rate of nearly \(O(N^{-3/4})\).

\begin{figure}[t]
\centering
\includegraphics[width=\textwidth]
{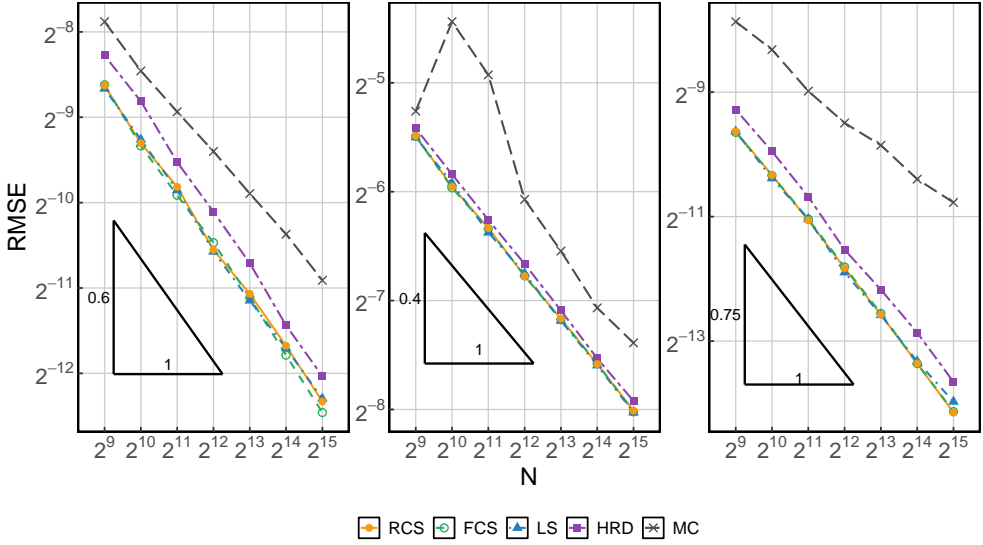}
\caption{Plots of RMSEs for $g_1$ (left), $g_2$ (middle), and $g_3$ (right). Each RMSE is computed from 448 independent replicates.}
\label{fig:nonsmooth-rmse}
\end{figure}

Figure~\ref{fig:nonsmooth-rmse} reports the RMSEs for the three integrands. All four RQMC methods significantly outperform MC and achieve the predicted convergence rates. The three Sobol'-based methods perform similarly, whereas HRD shows slightly larger errors.

\section{Discussions}\label{sec:discussion}

In this work, we have shown that for reduced GNS under coarse scrambling, the median RQMC estimator attains a comparable RMSE to MC, while achieving near-optimal convergence rates with high probability for function classes beyond \(L^2\). These results also hold for the completely random design \cite{pan2026automatic} and the Hankel random design \cite{goda2026quasi}, since their maximal gain coefficients are equal to \(1\). In particular, the theorems in Section~\ref{sec:main} hold verbatim with \(C_\beta\) replaced by \(1\).

As observed in Figure~\ref{fig:fc-boxplots}, coarsely scrambled Sobol' sequences consistently outperform the Hankel random design when coordinates are ordered in descending importance. Explaining this advantage requires exploiting the equidistribution property of Sobol' sequences. In the proof of Theorem~\ref{thm:median-weighted-probability}, for instance, our choice of the threshold \(T_u\) is conservative: many low-frequency \(\bsh\in K_u(T_u)\) actually satisfy \(\mathbb P\{\bm h\in D_{\bm C,m}^{\perp}(\bm M)\}=0\) due to equidistribution and can therefore be excluded from the union bound \eqref{eqn:unionbound}. Pan \cite{pan2026dimension} used this strategy for linearly scrambled digital nets to derive improved tractability conditions. We leave a full exploration of this idea in the present setting to future work.

Finally, although the error bounds in Sections~\ref{subsec:Lp}--\ref{subsec:Haarderiv} are probabilistic, an RMSE bound can be derived as follows. Suppose that for some constants \(C,\kappa>0\), the inequality
\[
\mathbb P\left\{
\left|\mathcal M_{m,R,\mathcal S}(f)-I(f)\right|
> C\delta^{-\kappa}
\right\}
\le
\frac12(8\delta)^{\frac{R+1}{2}}
\]
holds for all \(\delta\in(0,1/8)\). Then by the equality $\mathbb E|X|=\int_0^\infty \mathbb P\{|X|>x\}\rd x$,
\begin{align*}
 \mathbb E\left[
	\left|{\mathcal M}_{m,R,\mathcal S}(f)-I(f)\right|^2
	\right]
\le & (C 8^\kappa)^2+\int_{(C 8^\kappa)^2}^\infty  \mathbb P
\left\{\left|{\mathcal M}_{m,R,\mathcal S}(f)-I(f)\right|^2>x\right\} \rd x \\
\le & C^2\left(64^{\kappa}+\int_{0}^{1/8}\frac12(8\delta)^{\frac{R+1}{2}} \frac{2\kappa}{\delta^{2\kappa+1}} \rd\delta\right).
\end{align*}
For \(R\ge 4\kappa\), the integral converges, and the RMSE is bounded by a constant multiple of \(C\). As an example, applying this argument  to \eqref{eqn:Haarmedianbound} yields an RMSE of \(O(b^{-(\alpha_p+1/2-\varepsilon)m})\) for $f\in\mathcal H_{\mathrm{wav},\alpha,s,p,q}$ whenever \(R\ge 4\alpha_p+2\). Mean absolute error bounds can be obtained analogously.

\appendix

\section{Proof of Lemma~\ref{thm:walsh-error-scrambled-net}}
\label{app:proof-walsh-error-scrambled-net}

\begin{proof}
	Since $\bm D$ and $M$ are independent, it suffices to work conditionally on
	$M$. Thus, throughout the proof, we regard $M$ as fixed. First, we establish the result for Walsh polynomials. Let
	$f=\sum_{\bm h\in F}\widehat f(\bm h)\operatorname{wal}_{\bm h}$, where
	$F\subseteq\mathbb N_0^s$ is finite. The character property of digital nets \cite{dick:pill:2010}
	gives
	\begin{equation*}
	\frac1{b^m}\sum_{n=0}^{b^m-1}
	\operatorname{wal}_{\bm h}(\widetilde{\bm x}_n)
	=
	S_{\bm D}(\bm h)
	\mathbf1_{\{\bm h\in D_{\bm C,m}^{\perp}(M)\}},
	\end{equation*}
	where
	$S_{\bm D}(\bm h):=\prod_{j=1}^s
	\operatorname{wal}_{h_j}(\phi_b(D_j))$. Hence
	\begin{equation*}
	\frac1{b^m}\sum_{n=0}^{b^m-1}f(\widetilde{\bm x}_n)-I(f)
	=
	\sum_{\bm0\ne\bm h\in F}
	\widehat f(\bm h)S_{\bm D}(\bm h)
	\mathbf1_{\{\bm h\in D_{\bm C,m}^{\perp}(M)\}} .
	\end{equation*}
	Since $S_{\bm D}(\bm h)$ are orthonormal in
	$L^2(\bm D)$, taking the square and expectation with respect to $\bm D$
	yields
\begin{equation}\label{eqn:Walshpoly}
    \mathbb E_{\bm D}
	\left|
	\frac1{b^m}\sum_{n=0}^{b^m-1}f(\widetilde{\bm x}_n)-I(f)
	\right|^2
	=\sum_{\bm0\ne\bm h\in F\cap D_{\bm C,m}^{\perp}(M)}
	|\widehat{f}(\bm h)|^2.
\end{equation}

	For general $f\in L^2$, let $f_N$ be a sequence of Walsh polynomials with
	$f_N\to f$ in $L^2$ as $N\to\infty$. Since each $\widetilde{\bm x}_n$ is uniformly
	distributed with respect to $\bm D$, Jensen's inequality gives
	\begin{equation*}
	\mathbb E_{\bm D}
	\left|
	\frac1{b^m}\sum_{n=0}^{b^m-1}
	(f-f_N)(\widetilde{\bm x}_n)
	\right|^2
	\le
	\frac1{b^m}\sum_{n=0}^{b^m-1}
	\mathbb E_{\bm D}|(f-f_N)(\widetilde{\bm x}_n)|^2
	=
	\|f-f_N\|_{L^2}^2 .
	\end{equation*}
	Together with $|I(f-f_N)|\leq \|f-f_N\|_{L^2}$, the above bound yields
	\begin{align*}
	&\mathbb E_{\bm D}
	\left[
	\left|
	\left(\frac1{b^m}\sum_{n=0}^{b^m-1}f(\widetilde{\bm x}_n)-I(f)\right)
	-
	\left(\frac1{b^m}\sum_{n=0}^{b^m-1}f_N(\widetilde{\bm x}_n)-I(f_N)\right)
	\right|^2
	\right]\\
	&\qquad\le
	2\|f-f_N\|_{L^2}^2+2|I(f-f_N)|^2
	\le
	4\|f-f_N\|_{L^2}^2 .
	\end{align*}
	Moreover, by Parseval's identity,
	\begin{equation*}
	\sum_{\bm0\ne\bm h\in D_{\bm C,m}^{\perp}(M)}
	|\widehat{f-f_N}(\bm h)|^2
	\le
	\|f-f_N\|_{L^2}^2 .
	\end{equation*}
	Consequently, as $N\to\infty$, 
	\begin{align*}
	&\mathbb E_{\bm D}
	\left|
	\frac1{b^m}\sum_{n=0}^{b^m-1}f_N(\widetilde{\bm x}_n)-I(f_N)
	\right|^2
	\to
	\mathbb E_{\bm D}
	\left|
	\frac1{b^m}\sum_{n=0}^{b^m-1}f(\widetilde{\bm x}_n)-I(f)
	\right|^2
	,\\
	&\sum_{\bm0\ne\bm h\in D_{\bm C,m}^{\perp}(M)}
	|\widehat{f_N}(\bm h)|^2
	\to
	\sum_{\bm0\ne\bm h\in D_{\bm C,m}^{\perp}(M)}
	|\widehat f(\bm h)|^2 .
	\end{align*}
	Passing to the limit in the equality~\eqref{eqn:Walshpoly} proves the result.
\end{proof}

\section{Proof of Lemma~\ref{lem:HinL2}}\label{app:proof-lem:HinL2}

If 
$p\in [1,2]$, Jensen's inequality
$\sum_k a_k\le(\sum_k |a_k|^\lambda)^{1/\lambda}$ with
$\lambda=p/2\le1$ gives
\begin{equation*}
	\|Q_{\bm \ell}f\|_{L^2}=\left(
	\sum_{\bm k\in\Delta_{\bm \ell-\bm1}}
	\sum_{\bm i\in\nabla_{\bm \ell}}
	\left|
	\left\langle f,\psi_{\bm i,\bm k}^{\bm \ell}\right\rangle
	\right|^2
	\right)^{1/2}
	\le
	\left(
	\sum_{\bm k\in\Delta_{\bm \ell-\bm1}}
	\sum_{\bm i\in\nabla_{\bm \ell}}
	\left|
	\left\langle f,\psi_{\bm i,\bm k}^{\bm \ell}\right\rangle
	\right|^p
	\right)^{1/p}.
\end{equation*}
Therefore,
\begin{equation*}
	\|f\|^q_{\mathrm{wav},\alpha_p,s,2,q}=\sum_{\bm \ell\in\mathbb N_0^s}
	b^{q(\alpha-1/p+1/2)|\bm \ell|_1}
	\|Q_{\bm \ell}f\|_{L^2}^q
	\le
	\|f\|_{\mathrm{wav},\alpha,s,p,q}^{\,q}.
\end{equation*}
If $p>2$, \eqref{eqn:Hnormpq} gives
$\|f\|_{\mathrm{wav},\alpha,s,2,q}\le
\|f\|_{\mathrm{wav},\alpha,s,p,q}$. Hence, \eqref{Walsh_bound} holds in both cases. 

When $\alpha_p>0$ and $q\in[1,2]$, Jensen's inequality with $\lambda=q/2\le1$ gives
\begin{align*}
    \|f\|_{L^2}^{\,2}=\sum_{\bm \ell\in\mathbb N_0^s}\|Q_{\bm \ell}f\|_{L^2}^2\leq \left(\sum_{\bm \ell\in\mathbb N_0^s}\|Q_{\bm \ell}f\|_{L^2}^q\right)^{2/q}\leq \|f\|^2_{\mathrm{wav},\alpha_p,s,2,q}.
\end{align*}

When $\alpha_p>0$ and $q>2$, let $q':=q/(q-2)$. H\"older's inequality gives
\begin{align*}
    \|f\|_{L^2}^{\,2}=\sum_{\bm \ell\in\mathbb N_0^s}\|Q_{\bm \ell}f\|_{L^2}^2\leq & \left(\sum_{\bm \ell\in\mathbb N_0^s}b^{q\alpha_p|\bm \ell|_1}\|Q_{\bm \ell}f\|_{L^2}^q\right)^{2/q}\left(\sum_{\bm \ell\in\mathbb N_0^s}b^{-2q'\alpha_p|\bm \ell|_1}\right)^{1/q'}\\
    = & \|f\|^2_{\mathrm{wav},\alpha_p,s,2,q}(1-b^{-2q'\alpha_p})^{-s/q'}.
\end{align*}
Combining the two cases yields \eqref{eqn:L2embedding} since $\|f\|_{\mathrm{wav},\alpha_p,s,2,q}\le
\|f\|_{\mathrm{wav},\alpha,s,p,q}$.

\section{Proof of Lemma~\ref{lem:HaarWalshBlock}}
\label{app:proof-HaarWalshBlock}

\begin{proof}
We first prove the one-dimensional case.  For \(\ell=0\), \(Q_0=P_0\) is the integration projection, and \(\Lambda_0=\{0\}\), so the claim is immediate. For \(\ell\ge1\), set
\[
\mathcal W_\ell:=\operatorname{span}\{\operatorname{wal}_h: b^{\ell-1}\le h<b^\ell\}.
\]
We claim that \(\operatorname{ran}(Q_\ell)=\mathcal W_\ell\).

Indeed, for every integer \(h\in[b^{\ell-1},b^\ell)\), \(\operatorname{wal}_h\) has zero integral over every \(b\)-adic interval of length \(b^{-(\ell-1)}\), so \(\operatorname{wal}_h\in\ker P_{\ell-1}\); meanwhile, \(\operatorname{wal}_h\) is piecewise constant on intervals of length \(b^{-\ell}\), so \(\operatorname{wal}_h\in\operatorname{ran}P_\ell\). Therefore,
\[
\operatorname{wal}_h\in \operatorname{ran}P_\ell\cap\ker P_{\ell-1}=\operatorname{ran}Q_\ell,
\]
which gives \(\mathcal W_\ell\subseteq\operatorname{ran}(Q_\ell)\). The equality follows since both spaces have dimension \(b^\ell-b^{\ell-1}\). Thus the one-dimensional claim holds for every \(\ell\in\mathbb N_0\). The multidimensional result is obtained by taking tensor products over the \(s\) coordinates.
\end{proof}

\section{Proof of Lemma~\ref{lem:Ku-size-deriv}}\label{app:lem:Ku-size-deriv}

The proof requires the following one‑dimensional estimate.

\begin{lemma}\label{lem:count1d}
Let \(r\in\mathbb{N}\), \(\alpha_p\ge 0\), and set \(c:=r+\alpha_p\). For \(h\in\mathbb{N}\), define
\[
S(h):=(1-\alpha_p)\mu_r(h)+\alpha_p\mu_{r+1}(h).
\]
Then there exists a constant \(C_0\) depending on $r$ and $\alpha_p$ such that for every \(L\ge 0\),
\[
|\{h\in\mathbb{N}: S(h)\le L\}|\le C_0\, b^{\,L/c}.
\]
\end{lemma}

\begin{proof}
If \(h\) has at most \(r\) nonzero digits, then \(S(h)=\mu_r(h)\le L\), giving at most
\[
\sum_{j=1}^r \binom{L}{j}(b-1)^j \le rL^r(b-1)^r \le C_1 b^{L/c},
\]
where $C_1$ depends on $r$ and $c$. Otherwise, let \(q:=\mu_{r+1}(h)-\mu_r(h)\ge 1\). Since the first \(r\) nonzero positions are strictly larger than \(q\), we have \(\mu_r(h)\ge r(q+1)\). Thus
\[
S(h)=\mu_r(h)+\alpha_p q\ge (r+\alpha_p)q+r=cq+r.
\]
For fixed \(q\le (L-r)/c\), the digits in the first \(q\) positions have at most \(b^q\) possible values; the remaining \(r\) nonzero positions can be chosen from at most \((L-cq)^r\) possibilities, and each such position has at most \(b-1\) digit values, yielding a total number $b^q (L-cq)^r (b-1)^r$. Summing over \(q\) and setting \(L'=L-cq\),
\[
\sum_{q=1}^{\lfloor (L-r)/c\rfloor} b^q (L-cq)^r (b-1)^r \leq (b-1)^r\sum_{L'=r}^{\infty} b^{(L-L')/c} (L')^r \le C_2 b^{L/c},
\]
where $C_2$ depends on $r$ and $c$. Taking \(C_0=C_1+C_2\) completes the proof.
\end{proof}

\begin{proof}[Proof of Lemma~\ref{lem:Ku-size-deriv}]
If \(T<|u|\), the set \(K_{u,r,\alpha_p}(T)=\varnothing\) is empty  and the bound is trivial. Assume henceforth \(T\ge |u|\).
For each \(j\in u\), set \(L_j=S(h_j)\). By Lemma~\ref{lem:count1d},
\[
|K_{u,r,\alpha_p}(T)|
\le C_0^{|u|}
\sum_{\substack{L_j\ge 1,\; j\in u\\ \sum_{j\in u}L_j\le T}}
b^{\sum_{j\in u}L_j/c}\leq C_0^{|u|}
\sum_{n=|u|}^{\lfloor T\rfloor}
b^{n/c}\binom{n-1}{|u|-1}.
\]
Applying \eqref{eqn:cardbound} with $b$ replaced by $b^{1/c}$ yields the claimed bound.
\end{proof}

\section{Proof of Lemma \ref{lem:walsh-tail-outside-Ku-deriv}}\label{app:liu-block-proof}

We first recall the auxiliary function introduced in \cite{Suzuki2016Walsh}. For
$k=0$, set $W(0)(x):=1$. For $k>0$, write its $b$-adic expansion uniquely as $k=\sum_{i=1}^{\nu}c_i b^{a_i-1}$ with
$c_i\in\mathbb Z_b\setminus\{0\}$ and $a_1>\cdots>a_\nu\ge1$. Define
recursively
\begin{equation*}
	W(k)(x)
	:=
	\int_0^x
	\overline{\operatorname{wal}_{c_\nu b^{a_\nu-1}}(y)}
	W(k-c_\nu b^{a_\nu-1})(y)\,\rd y .
\end{equation*}
It is known that $W(k)\in L^\infty$ and, if $\nu>0$, $W(k)$ is periodic
with period $b^{-a_\nu+1}$ \cite{Suzuki2016Walsh}.
For $\bm k\in\mathbb N_0^s$, we use the tensor-product notation
\begin{equation*}
	W(\bm k)(\bm x):=\prod_{j=1}^sW(k_j)(x_j).
\end{equation*}

Next, let \(r\in\mathbb N\). For \(h\in\mathbb N\) with the above $b$-adic expansion, define
\[
h_r^+:=\sum_{i=1}^{\min(r,\nu)}c_i b^{a_i-1},\qquad
h_r^-:=h-h_r^+.
\]
Also set
\[
h_{(i)}:=
\begin{cases}
a_r,&1\le i\le\nu,\\
0,&\text{otherwise},
\end{cases}
\qquad
n_r(h):=\min(r,\nu).
\]
For \(h=0\), these quantities are defined to be zero. For a vector \(\bm h=(h_1,\dots,h_s)\in\mathbb N_0^s\), the notation \(\bm h_r^+\), \(\bm h_r^-\), \(\bm h_{(r)}\), and \(\bm n_r(\bm h)\) is understood componentwise.

The following lemma generalizes \cite[Theorem~2.5]{Suzuki2016Walsh} to Sobolev space settings.

\begin{lemma}\label{lem:suzuki-yoshiki-sobolev}
	Let $r\in \natu$ and $f\in W_{\mathrm{mix}}^{r,1}$.
	Then, for every $\bm h\in\mathbb N_0^s$,
	\begin{equation}\label{eq:liu-ibp-formula}
		\widehat f(\bm h)
		=
		(-1)^{|\bm n_r(\bm h)|_1}
		\int_{(0,1)^s}
		f^{(\bm n_r(\bm h))}(\bm x)
		\overline{\operatorname{wal}_{\bm h_r^-}(\bm x)}
		W(\bm h_r^+)(\bm x)
		\,\rd\bm x .
	\end{equation}
\end{lemma}

\begin{proof}
	Let $\nu(k)$ denote the number of nonzero digits in the $b$-adic
	expansion of $k$, with $\nu(0):=0$. Following \cite{Suzuki2016Walsh}, define
	the iterated integrals of the conjugate Walsh function by
	\begin{equation*}
		\overline{\operatorname{wal}}_k^{[0]}(x)
		:=
		\overline{\operatorname{wal}_k(x)},
		\qquad
		\overline{\operatorname{wal}}_k^{[n]}(x)
		:=
		\int_0^x
		\overline{\operatorname{wal}}_k^{[n-1]}(y)\,\rd y .
	\end{equation*}
	By \cite[Lemma~2.3]{Suzuki2016Walsh}, for
	$0\le n\le\nu(k)$,
	\begin{equation}\label{eq:integrated-walsh-function}
		\overline{\operatorname{wal}}_k^{[n]}(x)
		=
		\overline{\operatorname{wal}_{k_n^-}(x)}W(k_n^+)(x),
	\end{equation}
	and, whenever $1\le n\le\nu(k)$,
	\begin{equation}\label{eq:integrated-walsh-boundary}
		\overline{\operatorname{wal}}_k^{[n]}(0)
		=
		\overline{\operatorname{wal}}_k^{[n]}(1)
		=0.
	\end{equation}

	We first establish the one-dimensional formula. Let
	$g\in W^{n,1}(0,1)$, where $0\le n\le\nu(k)$. If $n=0$, the desired
	identity is the definition of $\widehat g(k)$. Suppose that $n\ge1$.
	For every $0\le j<n$, we have $g^{(j)}\in W^{1,1}(0,1)$. By the
	one-dimensional Sobolev representation theorem
	\cite[Theorem~8.2]{brezis2011functional}, there is an absolutely continuous
	representative $\widetilde g^{(j)}$ on $[0,1]$ such that
	\begin{equation*}
		g^{(j)}=\widetilde g^{(j)}\quad\text{a.e. on }(0,1),
		\qquad
		\widetilde g^{(j)}(x)-\widetilde g^{(j)}(y)
		=
		\int_y^x g^{(j+1)}(t)\,\rd t,
		\quad x,y\in[0,1].
	\end{equation*}
	Denote the representative of $g$ by $\widetilde g$. The first integration
	by parts gives
	\begin{align*}
		\widehat g(k)
		&=
		\int_{(0,1)}\widetilde g(x)
		\overline{\operatorname{wal}}_k^{[0]}(x)\,\rd x\\
		&=
		\widetilde g(1)\overline{\operatorname{wal}}_k^{[1]}(1)
		-
		\widetilde g(0)\overline{\operatorname{wal}}_k^{[1]}(0)
		-
		\int_{(0,1)}g'(x)
		\overline{\operatorname{wal}}_k^{[1]}(x)\,\rd x
		\\
		&=
		-\int_{(0,1)}g'(x)
		\overline{\operatorname{wal}}_k^{[1]}(x)\,\rd x,
	\end{align*}
	where the boundary terms vanish by
	\eqref{eq:integrated-walsh-boundary}. Repeating the argument gives
	\begin{equation*}
		\widehat g(k)
		=
		(-1)^n\int_{(0,1)}
		g^{(n)}(x)\overline{\operatorname{wal}}_k^{[n]}(x)\,\rd x.
	\end{equation*}
	Together with \eqref{eq:integrated-walsh-function}, this yields
	\begin{equation}\label{eq:one-dimensional-sobolev-walsh}
		\widehat g(k)
		=
		(-1)^n\int_{(0,1)}
		g^{(n)}(x)
		\overline{\operatorname{wal}_{k_n^-}(x)}W(k_n^+)(x)\,\rd x.
	\end{equation}

	Now we turn to the multidimensional case. For each $j\in 1{:}s$, set $n_j:=n_r(h_j)$. We first record the
	relevant section regularity of mixed Sobolev functions.
By \cite[Theorem~2.1.4]{ziemer1989weakly}, we may choose a
representative $\widetilde f$ such that for almost every $\bm{x}_{-j}\in(0,1)^{s-1}$ the
section $t\mapsto \widetilde f(t,\bm{x}_{-j})$ belongs to $W^{1,1}(0,1)$,
and its weak derivative equals the section of $\partial_j f$ for almost every $t\in (0,1)$. Since $\partial_j f\in W^{1,1}(0,1)$, applying the theorem again to
$\partial_j f$ yields a representative $\widetilde{\partial_j f}$ with
analogous properties. By Fubini's theorem, for almost every
$\bm{x}_{-j}$, the functions
$t\mapsto \partial_j f(t,\bm{x}_{-j})$ and
$t\mapsto \widetilde{\partial_j f}(t,\bm{x}_{-j})$ agree for almost every $t$.
Therefore, for such $\bm{x}_{-j}$, the first weak derivative of
the section $t\mapsto \widetilde f(t,\bm{x}_{-j})$ is itself in $W^{1,1}(0,1)$, and its weak
derivative equals the section of $\partial_j^2 f$. Inductively, we
obtain that for almost every $\bm{x}_{-j}$, the section of $\widetilde f$ is in
$W^{n_j,1}(0,1)$ and its $n$-th weak derivative is the section of
$\partial_j^n f$ for $0\le n\le n_j$.
	The same statement holds with $f$ replaced by any mixed partial
	derivative of $f$ of order at most $r$ in the remaining variables.

	We now apply the one-dimensional identity
	\eqref{eq:one-dimensional-sobolev-walsh} successively. Starting from
	the definition of the Walsh coefficient and using Fubini's theorem,
	we may write
	\[
		\widehat f(\bm h)
		=
		\int_{(0,1)^{s-1}}
		\left(
			\int_0^1 f(x_1,\bm{x}_{2:s})
			\overline{\operatorname{wal}_{h_1}(x_1)}\,\mathrm{d}x_1
		\right)
		\prod_{j=2}^s
		\overline{\operatorname{wal}_{h_j}(x_j)}
		\,\mathrm{d}\bm{x}_{2:s}.
	\]
	By the section regularity just established, for almost every fixed
	$\bm{x}_{-1}$ the inner integrand is a function of $x_1$ in
	$W^{n_1,1}(0,1)$. Applying
	\eqref{eq:one-dimensional-sobolev-walsh} with $k=h_1$ and $n=n_1$
	gives
	\[
		\widehat f(\bm h)
		=
		(-1)^{n_1}
		\int_{(0,1)^s}
		\partial_1^{n_1} f(\bm x)
		\overline{\operatorname{wal}_{(h_1)_{n_1}^-}(x_1)}
		W((h_1)_{n_1}^+)(x_1)
		\prod_{j=2}^s
		\overline{\operatorname{wal}_{h_j}(x_j)}
		\,\mathrm{d}\bm x.
	\]

	Now $\partial_1^{n_1} f$ still belongs to the mixed Sobolev space of order $r$ in the remaining variables. Hence the section regularity
	statement applies to $\partial_1^{n_1} f$ in the $x_2$-direction. 
	Continuing in this way through $j=2,\dots,s$, we arrive at
	\[
		\widehat f(\bm h)
		=
		(-1)^{|\bm n_r(\bm h)|_1}
		\int_{(0,1)^s}
		f^{(\bm n_r(\bm h))}(\bm x)
		\prod_{j=1}^s
		\overline{\operatorname{wal}_{(h_j)_{n_j}^-}(x_j)}
		W((h_j)_{n_j}^+)(x_j)
		\,\mathrm{d}\bm x.
	\]
	By the definition of $\bm h_r^-$ and $\bm h_r^+$, this is exactly
	\eqref{eq:liu-ibp-formula}.
\end{proof}

Next, we introduce the Walsh blocks used in the subsequent analysis. For \(\boldsymbol{\ell}\in\mathbb N_0^s\), define the set of upper parts
\[
\mathcal A_{r,\boldsymbol{\ell},s}
:=
\left\{
\boldsymbol{k}\in\mathbb N_0^s\setminus\{\bm 0\}:
k_{j,(r+1)}=0\ \forall j\in 1{:}s,\quad
k_{j,(r)}>\ell_j\ \forall j\in \supp(\bm \ell)
\right\}.
\]
We call its elements admissible for \(\boldsymbol{\ell}\). By the definition, it holds for every \(\boldsymbol{k}\in\mathcal A_{r,\boldsymbol{\ell},s}\) that \(W(k_j)\) has period \(b^{-\ell_j}\) for all \(j\in 1{:}s\).
For an upper part \(\boldsymbol{k}_r^+\), define its set of admissible lower levels by
\[
\mathcal L_r(\boldsymbol{k}_r^+)
:=
\left\{
\boldsymbol{\ell}\in\mathbb N_0^{s}: 
\boldsymbol{k}_r^+\in
\mathcal A_{r,\boldsymbol{\ell},s}
\right\}.
\]
For \(\boldsymbol{k}_r^+\in\mathcal A_{r,\boldsymbol{\ell},s}\), define the associated Walsh block
\[
B_{r+1,\boldsymbol{\ell},s}(\boldsymbol{k}_r^+)
:=
\left\{
\boldsymbol{h}\in\mathbb N_0^s:
\boldsymbol{h}_r^+=\boldsymbol{k}_r^+,\ 
\Vec{\bm\mu}_1(\boldsymbol{h}_r^-)=\boldsymbol{\ell}
\right\}.
\]
Equivalently, for each \(j\in 1{:}s\), \(h_{j}\) ranges over the integers
\[
k_{j,r}^+ +\lfloor b^{\ell_j-1}\rfloor \le h_j \leq k_{j,r}^+ + b^{\ell_j}-1.
\]

\begin{lemma}\label{lem:liu-parseval-projection}
Let \(r\in\mathbb N\) and \(f\in W_{\mathrm{mix}}^{r,2}\) with ANOVA components \(f_u\). For \(\bm\ell\in\mathbb N_0^s\) and \(\bm k_r^+\in\mathcal A_{r,\bm\ell,s}\), set \(\bm\tau=\bm n_r(\bm k_r^+)\) and \(u=\operatorname{supp}(\bm k_r^+)\). Then
\begin{equation*}
\sum_{\bm h\in B_{r+1,\bm\ell,s}(\bm k_r^+)}
|\widehat f(\bm h)|^2
=
\left\|
Q_{\bm\ell_u}
\left(
f_u^{(\bm\tau_u)} W(\bm k_{u,r}^+)
\right)
\right\|_{L^2}^{2}.
\end{equation*}
\end{lemma}

\begin{proof}
For any \(\bm h\in B_{r+1,\bm\ell,s}(\bm k_r^+)\), we have \(\supp(\bm h)=\supp(\bm k_r^+)=u\) and \(\bm h_r^+=\bm k_r^+\). By \cite[Section~3]{pan2026dimension}, \(\widehat f(\bm h)=\widehat{f_u}(\bm h_u)\). Moreover, since \(\bm n_r(\bm k_r^+)=\bm\tau=(\bm\tau_u,\bm0_{-u})\), we have \(\bm n_r(\bm h_u)=\bm\tau_u\). Applying Lemma~\ref{lem:suzuki-yoshiki-sobolev} to \(\widehat{f_u}(\bm h_u)\) yields
\[
\widehat f(\bm h)
=
\widehat{f_u}(\bm h_u)
=
(-1)^{|\bm\tau_u|_1}
\int_{(0,1)^{|u|}}
f_u^{(\bm\tau_u)}(\bm x_u)
\overline{\operatorname{wal}_{\bm h_{u,r}^-}(\bm x_u)}
W(\bm k_{u,r}^+)(\bm x_u)
\,d\bm x_u .
\]
Thus, up to sign, \(\widehat f(\bm h)\) is the Walsh coefficient of \(f_u^{(\bm\tau_u)}W(\bm k_{u,r}^+)\) at frequency \(\bm h_{u,r}^-\). As \(\bm h\) ranges over \(B_{r+1,\bm\ell,s}(\bm k_r^+)\), the vector \(\bm h_{u,r}^-\) runs over the Walsh block
$\Lambda_{\bm{\ell}_u}=\{\bm \ell'_u\in\mathbb N_0^{|u|}:\Vec{\bm\mu}_1(\bm \ell'_u)=\bm\ell_u\}.$
The claim now follows from Lemma~\ref{lem:HaarWalshBlock}.
\end{proof}

\begin{lemma}\label{lem:periodic-W-filters-haar}
Let \(d\in\mathbb N\), \(\boldsymbol{\ell}\in\mathbb N_0^d\), and let \(W\in L^\infty(0,1)^d\) be periodic with period \(b^{-\ell_j}\) in the \(j\)-th coordinate for each \(j\in 1{:}d\). Then, for every \(g\in L^2(0,1)^d\),
\begin{equation*}
\|Q_{\boldsymbol{\ell}}(gW)\|_{L^2}^2
\le
\|W\|_{L^\infty}^2
\sum_{\substack{\boldsymbol{\ell}'\in\mathbb N_0^d\\\boldsymbol{\ell}'\ge\boldsymbol{\ell}}}
\|Q_{\boldsymbol{\ell}'}g\|_{L^2}^2 .
\end{equation*}
\end{lemma}

\begin{proof}
We first prove that $Q_{\boldsymbol{\ell}}\bigl((Q_{\boldsymbol{\ell}'}g)W\bigr)=0$
whenever \(\ell'_j<\ell_j\) for at least one coordinate \(j\). Fix such a \(j\). Since \(Q_{\boldsymbol{\ell}'}=\bigotimes_{r=1}^d(P_{\ell'_r}-P_{\ell'_r-1})\), the condition \(\ell'_j<\ell_j\) implies that, for almost every fixed \(\boldsymbol{x}_{-j}\), the section
$
x_j \mapsto (Q_{\boldsymbol{\ell}'}g)(x_j,\boldsymbol{x}_{-j})
$
is constant on every \(b\)-adic interval of length \(b^{-(\ell_j-1)}\).

We show that every Haar coefficient at level \(\boldsymbol{\ell}\) vanishes. Fix a tensor-product Haar function \(\psi_{\boldsymbol{a},\boldsymbol{k}}^{\boldsymbol{\ell}}\). In the \(j\)-th coordinate, its parent interval is
\[
I=[k_j b^{-(\ell_j-1)},(k_j+1)b^{-(\ell_j-1)}),
\]
and its active child is
\[
I_{a_j}=[(b k_j+a_j)b^{-\ell_j},(b k_j+a_j+1)b^{-\ell_j}).
\]
For fixed \(\boldsymbol{x}_{-j}\), let \(c(\boldsymbol{x}_{-j})\) be the constant value of \(x_j \mapsto (Q_{\boldsymbol{\ell}'}g)(x_j,\boldsymbol{x}_{-j})\) on \(I\). In the inner product \(\langle (Q_{\boldsymbol{\ell}'}g)W,\psi_{\boldsymbol{a},\boldsymbol{k}}^{\boldsymbol{\ell}}\rangle\), formula \eqref{formula_Haar} yields the following \(x_j\)-integral:
\[
c(\boldsymbol{x}_{-j})
\left[
b^{\ell_j/2}
\int_{I_{a_j}} W(x_j,\boldsymbol{x}_{-j})\,dx_j
-
b^{(\ell_j-2)/2}
\int_I W(x_j,\boldsymbol{x}_{-j})\,dx_j
\right].
\]
Since \(W\) has period \(b^{-\ell_j}\) in the \(j\)-th coordinate, we have
\[
\int_I W\,dx_j = b\int_{I_{a_j}} W\,dx_j.
\]
The displayed expression is therefore zero. Hence \(\langle (Q_{\boldsymbol{\ell}'}g)W,\psi_{\boldsymbol{a},\boldsymbol{k}}^{\boldsymbol{\ell}}\rangle=0\) for every \(\psi_{\boldsymbol{a},\boldsymbol{k}}^{\boldsymbol{\ell}}\), and $Q_{\boldsymbol{\ell}}\bigl((Q_{\boldsymbol{\ell}'}g)W\bigr)=0$ by \eqref{formula_Q}.

Consequently,
\begin{align*}
\|Q_{\boldsymbol{\ell}}(gW)\|_{L^2}
&=
\left\|Q_{\boldsymbol{\ell}}
\left(
\sum_{\substack{\boldsymbol{\ell}'\ge\boldsymbol{\ell}}}
(Q_{\boldsymbol{\ell}'}g)W
\right)
\right\|_{L^2} \\
&\le  
\left\|
\sum_{\substack{\boldsymbol{\ell}'\ge\boldsymbol{\ell}}}
(Q_{\boldsymbol{\ell}'}g)W
\right\|_{L^2} 
\le
\|W\|_{L^\infty}
\left\|
\sum_{\substack{\boldsymbol{\ell}'\ge\boldsymbol{\ell}}}
Q_{\boldsymbol{\ell}'}g
\right\|_{L^2}.
\end{align*}
The conclusion follows since the functions \(Q_{\boldsymbol{\ell}'}g\) are mutually orthogonal in \(L^2\).
\end{proof}

\begin{lemma}\label{lem:fixed-upper-joint-tail}
Under the assumptions of Lemma~\ref{lem:liu-parseval-projection}, suppose \(f_u^{(\bm\tau_u)}\in\mathcal H_{\mathrm{wav},\alpha,|u|,p,q}\) with \(\alpha_p>0\). Then, for every \(S\ge 0\),
\begin{equation}\label{eq:fixed-upper-joint-tail}
    \begin{aligned}
        &\sum_{\substack{\bm\ell\in\mathcal L_r(\bm k_r^+)\\
|\bm\ell|_1\ge S}}
\ \sum_{\bm h\in
B_{r+1,\bm \ell,s}(\bm k_r^+)}
|\widehat f(\bm h)|^2\\
\le
&D_{\alpha_p,q}^{|u|}
\,\bigl(\mathcal E_{|u|-1}(\lceil S\rceil)\bigr)^{(1-2/q)_+}
\,b^{-2\alpha_p S}
\,\|W(\bm k_{u,r}^+)\|_{L^\infty}^2
\,\|f_u^{(\bm\tau_u)}\|_{\mathrm{wav},\alpha,|u|,p,q}^{2}, 
    \end{aligned}
\end{equation}
where \(D_{\alpha_p,q}\) is a constant depending on \(\alpha_p\) and \(q\), and
\(\mathcal E_n(x):=\sum_{\nu=0}^n x^\nu/\nu!\).
\end{lemma}

\begin{proof}
By Lemma~\ref{lem:liu-parseval-projection}, the left-hand side of \eqref{eq:fixed-upper-joint-tail} equals
\[
\sum_{\substack{\bm\ell\in\mathcal L_r(\bm k_r^+)\\
|\bm\ell|_1\ge S}}
\left\|
Q_{\bm\ell_u}
\bigl(f_u^{(\bm\tau_u)} W(\bm k_{u,r}^+)\bigr)
\right\|_{L^2}^{2}.
\]

For each \(\bm\ell\in\mathcal L_r(\bm k_r^+)\), the definition of \(\mathcal A_{r,\bm\ell,s}\) implies that \(W(\bm k_{u,r}^+)\) has period \(b^{-\ell_j}\) in the \(j\)-th coordinate for every \(j\in u\). Applying Lemma~\ref{lem:periodic-W-filters-haar} gives
\[
\|Q_{\bm\ell_u}(f_u^{(\bm\tau_u)}W(\bm k_{u,r}^+))\|_{L^2}^2
\le
\|W(\bm k_{u,r}^+)\|_{L^\infty}^2
\sum_{\substack{\bm{\ell}'_u\in\mathbb N_0^{|u|}\\\bm{\ell}'_u\ge\bm\ell_u}}
\|Q_{\bm{\ell}'_u}f_u^{(\bm\tau_u)}\|_{L^2}^2 .
\]
Summing over \(\bm\ell\in\mathcal L_r(\bm k_r^+)\) with \(|\bm\ell|_1\ge S\) and interchanging the sums yields
\[
\sum_{\substack{\bm\ell\in\mathcal L_r(\bm k_r^+)\\
|\bm\ell|_1\ge S}}
\|Q_{\bm\ell_u}(f_u^{(\bm\tau_u)}W(\bm k_{u,r}^+))\|_{L^2}^2
\le
\|W(\bm k_{u,r}^+)\|_{L^\infty}^2
\sum_{\bm{\ell}'_u\in\mathbb N_0^{|u|}}
H(\bm{\ell}'_u)\,
\|Q_{\bm{\ell}'_u}f_u^{(\bm\tau_u)}\|_{L^2}^2 ,
\]
where
\[
H(\bm{\ell}'_u)
:=
\bigl|\bigl\{\bm\ell\in\mathcal L_r(\bm k_r^+):
\bm\ell_u\le\bm{\ell}'_u,\ |\bm\ell|_1\ge S\bigr\}\bigr|.
\]

Since \(\supp(\bm\ell)\subseteq u\) for every \(\bm\ell\in\mathcal L_r(\bm k_r^+)\), we have \(|\bm\ell|_1=|\bm\ell_u|_1\). If \(|\bm{\ell}'_u|_1<S\), then \(H(\bm{\ell}'_u)=0\). For \(|\bm{\ell}'_u|_1=\lceil S\rceil+t\) with \(t\ge0\), setting \(\bm{\Delta\ell}_u:=\bm{\ell}'_u-\bm\ell_u\) gives
\begin{equation}\label{eqn:Hl}
H(\bm{\ell}'_u)\le \bigl|\{\bm{\Delta\ell}_u\in\mathbb N_0^{|u|}:|\bm{\Delta\ell}_u|_1\le t\}\bigr|
=\binom{t+|u|}{|u|}.
\end{equation}

We now distinguish two cases.

\textbf{Case 1: \(q\in[1,2]\).} Substituting \eqref{eqn:Hl} yields
\begin{align*}
&\sum_{\bm{\ell}'_u\in\mathbb N_0^{|u|}}H(\bm{\ell}'_u)\|Q_{\bm{\ell}'_u}f_u^{(\bm\tau_u)}\|_{L^2}^2\\
		\le & \sum_{t=0}^\infty\sum_{\substack{\bm{\ell}'_u\in\mathbb N_0^{|u|}\\ |\bm{\ell}'_u|_1=\lceil S\rceil +t}}\binom{t+|u|}{|u|}\|Q_{\bm{\ell}'_u}f_u^{(\bm\tau_u)}\|_{L^2}^2
		\\
		= & b^{-2\alpha_p \lceil S\rceil}
		\sum_{t=0}^{\infty}\binom{t+|u|}{|u|}b^{-2\alpha_p t}
		\sum_{\substack{\bm{\ell}'_u\in\mathbb N_0^{|u|}\\ |\bm{\ell}'_u|_1=\lceil S\rceil +t}}b^{2\alpha_p |\bm{\ell}'_u|_1}\|Q_{\bm{\ell}'_u}f_u^{(\bm\tau_u)}\|_{L^2}^2
		\\
		\le & b^{-2\alpha_p S} \|f_u^{(\bm\tau_u)}\|_{\mathrm{wav},\alpha,|u|,p,2}^{2}	\sum_{t=0}^{\infty}\binom{t+|u|}{|u|}b^{-2\alpha_p t}\\
        \leq &	.b^{-2\alpha_p S} \|f_u^{(\bm\tau_u)}\|_{\mathrm{wav},\alpha,|u|,p,q}^{2}	\left(1-b^{-2\alpha_p}\right)^{-|u|-1},
\end{align*}
where the last inequality follows from \eqref{eqn:Hnormpq} and \cite[Lemma 13.24]{dick:pill:2010}.

\textbf{Case 2: \(q>2\).} Let \(q':=q/(q-2)\). An argument analogous to \eqref{eqn:LqLq'} gives
\begin{equation}\label{eqn:caseqlarger2}
\sum_{\bm{\ell}'_u\in\mathbb N_0^{|u|}}H(\bm{\ell}'_u)\|Q_{\bm{\ell}'_u}f_u^{(\bm\tau_u)}\|_{L^2}^2
\le
\|f_u^{(\bm\tau_u)}\|_{\mathrm{wav},\alpha,|u|,p,q}^{2}
\Biggl(\sum_{\bm{\ell}'_u\in\mathbb N_0^{|u|}}
H(\bm{\ell}'_u)^{q'} b^{-2q' \alpha_p |\bm{\ell}'_u|_1}
\Biggr)^{1/q'}.
\end{equation}

The number of vectors \(\bm{\ell}'_u\) with \(|\bm{\ell}'_u|_1=\lceil S\rceil+t\) is \(\binom{\lceil S\rceil+t+|u|-1}{|u|-1}\). By Vandermonde's identity,
\[
\binom{\lceil S\rceil+t+|u|-1}{|u|-1}
=
\sum_{\nu=0}^{|u|-1}
\binom{\lceil S\rceil}{\nu}
\binom{t+|u|-1}{|u|-1-\nu}.
\]
Using \eqref{eqn:Hl}, we obtain
	\[
	\begin{aligned}
		&\sum_{\bm{\ell}'_u\in\mathbb N_0^{|u|}}H(\bm{\ell}'_u)^{q'} b^{-2q' \alpha_p |\bm{\ell}'_u|_1}\\
		\le& b^{-2q' \alpha_p \lceil S\rceil}
		\sum_{\nu=0}^{|u|-1}\binom{\lceil S\rceil}{\nu}
		\sum_{t=0}^\infty
		\binom{t+|u|-1}{|u|-1-\nu}
		\binom{t+|u|}{|u|}^{q'}b^{-2q' \alpha_p t}.
	\end{aligned}
	\]
Choose \(z>0\) small enough that \((1+z)^{1+q'}b^{-2q'\alpha_p}<1\). For each fixed \(0\le\nu\le |u|-1\), the binomial theorem yields
	\begin{align*}
	&\sum_{t=0}^\infty
	\binom{t+|u|-1}{|u|-1-\nu}
	\binom{t+|u|}{|u|}^{q'}b^{-2q'\alpha_pt}\\
	\leq &\sum_{t=0}^\infty
	\left[z^{-(|u|-1-\nu)}(1+z)^{t+|u|-1}\right]
	\left[z^{-|u|}(1+z)^{t+|u|}\right]^{q'}
	b^{-2q'\alpha_pt}\\
	=&
	z^{-(|u|-1-\nu)-|u|q'}(1+z)^{|u|-1+|u|q'}
	\sum_{t=0}^\infty
	\left((1+z)^{1+q'}b^{-2q'\alpha_p}\right)^t\\
	=&\frac{z^{1+\nu}}{(1+z)\left(1-(1+z)^{1+q'}b^{-2q'\alpha_p}\right)}
\left( \frac{1+z}{z} \right)^{|u|(1+q')}.
\end{align*}
Hence, for a sufficiently large constant \(D_{\alpha_p,q}\),
\[
\sum_{\bm{\ell}'_u\in\mathbb N_0^{|u|}}H(\bm{\ell}'_u)^{q'} b^{-2q' \alpha_p |\bm{\ell}'_u|_1}
\le
\frac{(D_{\alpha_p,q})^{q'|u|}}{b^{2q' \alpha_p \lceil S\rceil}}
\sum_{\nu=0}^{|u|-1}\binom{\lceil S\rceil}{\nu}
\le
\frac{(D_{\alpha_p,q})^{q'|u|}}{b^{2q' \alpha_p S}}
\sum_{\nu=0}^{|u|-1}\frac{(\lceil S\rceil)^\nu}{\nu!}.
\]
Substituting this bound into \eqref{eqn:caseqlarger2} proves \eqref{eq:fixed-upper-joint-tail} for \(q>2\). Increasing \(D_{\alpha_p,q}\) if necessary to cover \(q\in[1,2]\) completes the proof.
\end{proof}

\begin{proof}[Proof of Lemma~\ref{lem:walsh-tail-outside-Ku-deriv}]
Define
\[
S_{r,\alpha_p}(\bm h)
:=
(1-\alpha_p)\mu_r(\bm h)+\alpha_p\mu_{r+1}(\bm h).
\]
For \(\bm h\in B_{r+1,\bm\ell,s}(\bm k_r^+)\) with \(u=\supp(\bm k_r^+)\),
\begin{equation}\label{eq:rho-upper-lower}
S_{r,\alpha_p}(\bm h)
=
\mu_r(\bm k_r^+)+\alpha_p|\bm\ell_u|_1.
\end{equation}

Fix \(\bm\tau_u\in\{1{:}r\}^{|u|}\), and consider the admissible
upper parts \(\bm k_{u,r}^+\) with
\(\bm n_r(\bm k_{u,r}^+)=\bm\tau_u\). Write
\[
\sigma:=|\bm\tau_u|_1,
\qquad
M:=\mu_r(\bm k_{u,r}^+),
\qquad
\lambda:=(1-2/q)_+.
\]
For fixed \(M\), the number of admissible upper parts \(\bm k_{u,r}^+\)
satisfying \(\bm n_r(\bm k_{u,r}^+)=\bm\tau_u\) and
\(\mu_r(\bm k_{u,r}^+)=M\) is at most
\begin{equation}\label{eq:upper-part-count}
(b-1)^\sigma\binom{M-1}{\sigma-1}.
\end{equation}
Indeed, after ignoring the restrictions within the coordinates,
we choose the \(\sigma\) nonzero digit positions summing to $M$, and then their nonzero values. Moreover, \cite[Lemma~4]{goda2026quasi} gives
\begin{equation}\label{eq:W-upper-part-bound}
\|W(\bm k_{u,r}^+)\|_{L^\infty}^2
\le
\left(\frac{1+\pi}{2}\right)^{2|u|}
b^{2\sigma-2M}.
\end{equation}

We first consider upper parts with \(M\le T\), which exist only when \(T\ge\sigma\). By \eqref{eq:rho-upper-lower}, a block with such an upper part lies outside \(K_{u,r,\alpha_p}(T)\) exactly when \(\alpha_p|\bm\ell_u|_1>T-M\). We may assume \(\alpha_p>0\), since otherwise \(\alpha_p|\bm\ell_u|_1=0\) and the block is inside \(K_{u,r,\alpha_p}(T)\). Applying Lemma~\ref{lem:fixed-upper-joint-tail} with \(S=(T-M)/\alpha_p\), followed by \eqref{eq:W-upper-part-bound}, shows that the contribution of one such upper part is at most
\[
D_{\alpha_p,q}^{|u|}
\left(\frac{1+\pi}{2}\right)^{2|u|}
b^{2\sigma}\left(\mathcal  E_{|u|-1}(\lceil (T-M)/\alpha_p \rceil)\right)^\lambda b^{-2T}
\|f_u^{(\bm\tau_u)}\|_{\mathrm{wav},\alpha,|u|,p,q}^2.
\]
Multiplying by the count \eqref{eq:upper-part-count} and summing over \(\sigma\le M\le \lfloor T\rfloor\), we get
\begin{align*}
&\sum_{M=\sigma}^{\lfloor T\rfloor}
\binom{M-1}{\sigma-1}\left(\mathcal  E_{|u|-1}(\lceil (T-M)/\alpha_p \rceil)\right)^\lambda
\\
\le&
\max\left(1,\alpha^{-\lambda(|u|-1)}_p\right)\sum_{M=\sigma}^{\lfloor T\rfloor}
\binom{M-1}{\sigma-1}\left(\mathcal E_{|u|-1}(1+T-M)\right)^\lambda\\
\leq &
\max\left(1,\alpha^{-\lambda(|u|-1)}_p\right)\mathcal E_{r|u|}(T)\left(\mathcal E_{|u|-1}(T)\right)^\lambda,
\end{align*}
where the last inequality uses $\sum_{M=\sigma}^{\lfloor T\rfloor}
\binom{M-1}{\sigma-1}=\binom{\lfloor T\rfloor}{\sigma}\leq \mathcal E_{r|u|}(T)$ since $\sigma\leq r|u|$.
Thus, for fixed \(\bm\tau_u\), the total contribution of upper parts with \(M\le T\) satisfies
\begin{equation}\label{eq:upper-parts-le-T}
    \begin{aligned}
        \sum_{\substack{\bm h\in\mathbb N_0^{s}\setminus K_{u,r,\alpha_p}(T)\\
\supp(\bsh)=u,\bm n_r(\bm h_u)=\bm\tau_u,\ \mu_r(\bm h_u)\le T}}
|\widehat f(\bm h)|^2
\le 
D_{r,\alpha_p,q}^{|u|}
\frac{\mathcal E_{r|u|}(T)\left(\mathcal E_{|u|-1}(T)\right)^\lambda}{b^{2T}}
\|f_u^{(\bm\tau_u)}\|_{\mathrm{wav},\alpha,|u|,p,q}^2,
    \end{aligned}
\end{equation}
where 
$$D_{r,\alpha_p,q}=D_{\alpha_p,q}\left(\frac{1+\pi}{2}\right)^{2}b^{2r}(b-1)^{r}\max\left(1,\alpha^{-\lambda}_p\right).$$

It remains to handle upper parts with \(M>T\). For every \(\bm\ell\in\mathcal L_r(\bm k_r^+)\), equation \eqref{eq:rho-upper-lower} shows that every \(\bm h\in B_{r+1,\bm\ell,s}(\bm k_r^+)\) satisfies \(S_{r,\alpha_p}(\bm h)=M+\alpha_p|\bm\ell_u|_1>T\). Hence every such block lies outside \(K_{u,r,\alpha_p}(T)\). For a fixed upper part \(\bm k_r^+\), Lemma~\ref{lem:liu-parseval-projection} implies
\begin{align*}
&\sum_{\bm\ell\in\mathcal L_r(\bm k_r^+)}
\ \sum_{\bm h\in
B_{r+1,\bm \ell,s}(\bm k_r^+)}
|\widehat f(\bm h)|^2
=
\sum_{\bm\ell\in\mathcal L_r(\bm k_r^+)}
\|Q_{\bm\ell_u}(f_u^{(\bm\tau_u)} W(\bm k_{u,r}^+))\|_{L^2}^2
\\
\le &
\|f_u^{(\bm\tau_u)} W(\bm k_{u,r}^+)\|_{L^2}^2
\le
\|W(\bm k_{u,r}^+)\|_{L^\infty}^2
\|f_u^{(\bm\tau_u)}\|_{L^2}^2\le
\left(\frac{1+\pi}{2}\right)^{2|u|}
b^{2\sigma-2M}
\|f_u^{(\bm\tau_u)}\|_{L^2}^2,
\end{align*}
where the last inequality follows from
\eqref{eq:W-upper-part-bound}. 

Set $T_*=\max(\lfloor T\rfloor+1,\sigma)$.
Grouping upper parts by \(M\) and using \eqref{eq:upper-part-count}, we obtain
\begin{equation}\label{eq:upper-parts-gt-T}
    \begin{aligned}
        &\sum_{\substack{\bm h\in\mathbb N_0^{s}\setminus K_{u,r,\alpha_p}(T)\\
\supp(\bsh)=u,\bm n_r(\bm h_u)=\bm\tau_u,\ \mu_r(\bm h_u)> T}}
|\widehat f(\bm h)|^2
\\
\le & \left(\frac{1+\pi}{2}\right)^{2|u|}
\bigl(b^2(b-1)\bigr)^\sigma
\left(
\sum_{M=T_*}^\infty\binom{M-1}{\sigma-1}b^{-2M}
\right)
\|f_u^{(\bm\tau_u)}\|_{L^2}^2
\\
\le &
\left(\frac{1+\pi}{2}\right)^{2|u|}
\bigl(b^2(b-1)(1-b^{-2})\bigr)^{\sigma}
\mathcal E_{\sigma-1}(T)b^{-2T}
\|f_u^{(\bm\tau_u)}\|_{L^2}^2.
    \end{aligned},
\end{equation}
where the last inequality uses \cite[Lemma 13.24]{dick:pill:2010} to bound
\begin{equation*}
\sum_{M=T_*}^\infty
\binom{M-1}{\sigma-1}b^{-2M}
\le
b^{-2T_*}\binom{T_*-1}{\sigma-1}(1-b^{-2})^{-\sigma}\leq b^{-2T}\mathcal E_{\sigma-1}(T_+)(1-b^{-2})^{-\sigma}.
\end{equation*}
Finally, if $\alpha=0, p=q=2$, we substitute $\|f_u^{(\bm\tau_u)}\|_{L^2}=\|f_u^{(\bm\tau_u)}\|_{\mathrm{wav},\alpha,|u|,p,q}$ into \eqref{eq:upper-parts-gt-T}. If $\alpha_p>0$, we substitute the bound \eqref{eqn:L2embedding} for $\|f_u^{(\bm\tau_u)}\|_{L^2}$ into \eqref{eq:upper-parts-gt-T} and add \eqref{eq:upper-parts-le-T}. In both cases, we have
\begin{align*}
\sum_{\substack{\bm h\in\mathbb N_0^{s}\setminus K_{u,r,\alpha_p}(T)\\
\supp(\bsh)=u,\bm n_r(\bm h_u)=\bm\tau_u}}
|\widehat f(\bm h)|^2
\le 
 C_{r,\alpha_p,q}^{|u|}
\frac{\mathcal E_{r|u|}(T_+)\left(\mathcal E_{|u|-1}(T_+)\right)^\lambda}{b^{2T}}
\|f_u^{(\bm\tau_u)}\|_{\mathrm{wav},\alpha,|u|,p,q}^2
\end{align*}
for a sufficiently large constant \(C_{r,\alpha_p,q}\). The proof is complete upon summing over
\(\bm\tau_u\in\{1{:}r\}^{|u|}\) and bounding
\begin{align*}
   \sum_{\bm\tau_u\in\{1{:}r\}^{|u|}}
\|f_u^{(\bm\tau_u)}\|_{\mathrm{wav},\alpha,|u|,p,q}^2
\leq r^{|u|} 
\|f_u\|_{\mathrm{mix}^*,r,\alpha,|u|,p,q}^2. 
\end{align*}
\end{proof}

\section{Proof of Corollary~\ref{cor:deriv}}\label{app:cor:deriv}
\begin{proof}
Using an argument analogous to \eqref{eqn:tractabilitysum}, the sum \(\psi_{r,\alpha,s,p,q,\gamma}(m)\) is bounded by \(C_{\varepsilon,\Upsilon,r,\alpha,p,q} b^{\varepsilon m}\), with \(C_{\varepsilon,\Upsilon,r,\alpha,p,q}\) independent of \(s\) and \(m\).

It remains to bound $\psi'_{r,\alpha,s,p,q,\gamma}(m)$. 
Fix $\zeta>0$ and $n\in\mathbb N_0$, and set
$x=(r+\alpha_p)m$. Since $m\ge0$ implies $x\ge0$, expanding
the definition of $\mathcal E_n(x)$ gives
\begin{align*}
\sup_{m\ge0}b^{-\zeta m}
\mathcal E_n(x)
&\leq
\sup_{x\ge 0}
\sum_{\nu=0}^n\frac{x^\nu}{\nu!}
\exp\left(-\frac{\zeta\log b}{r+\alpha_p}x\right)\\
&\le
\sum_{\nu=0}^n\frac{1}{\nu!}
\sup_{x\ge0}
x^\nu\exp\left(-\frac{\zeta\log b}{r+\alpha_p}x\right)\\
&\overset{(1)}{\leq}
\sum_{\nu=0}^n
\left(\frac{r+\alpha_p}{\zeta\log b}\right)^\nu
\le \left(\frac{r+\alpha_p}{\zeta\log b}\right)^{n+1}
\end{align*}
if $ 2\zeta\log b\leq r+\alpha_p$, where (1) follows from
\[
\sup_{x\ge0}x^\nu e^{-cx}=\left(\frac{\nu}{ce}\right)^\nu
\quad\text{and}\quad
\nu!\ge\left(\frac{\nu}{e}\right)^\nu.
\]

Therefore, for sufficiently small $\varepsilon>0$, applying this bound with $\zeta=\varepsilon/2$ to both $\mathcal E_{r|u|}(x)$ and $\mathcal E_{|u|-1}(x)$ in
$\mathfrak P_{|u|,r,q}(x)$ yields
\begin{align*}
&b^{-\varepsilon m}\psi'_{r,\alpha,s,p,q,\gamma}(m)\\
\le &
\sum_{\emptyset\ne u\subseteq1{:}s}
C_{r,\alpha_p,q}^{|u|}|u|!
\left(\prod_{j\in u}\Upsilon_j^{\frac{1}{r+\alpha_p+1/2}}\right)
\,b^{-\varepsilon m/2}
\mathcal E_{r|u|}(x)
\left(b^{-\varepsilon m/2}\mathcal E_{|u|-1}(x)\right)^{(1-2/q)_+}\\
\le &
\sum_{\emptyset\ne u\subseteq1{:}s}
C_{r,\alpha_p,q}^{|u|}|u|!
\left(\prod_{j\in u}\Upsilon_j^{\frac{1}{r+\alpha_p+1/2}}\right)
 \left(\frac{2(r+\alpha_p)}{\varepsilon\log b}\right)^{r|u|+1+|u|(1-2/q)_+}\\
\le &
 \left(\frac{2(r+\alpha_p)}{\varepsilon\log b}\right)\sum_{\emptyset\ne u\subseteq1{:}s}
\left(C_{r,\alpha_p,q}\left(\frac{2(r+\alpha_p)}{\varepsilon\log b}\right)^{r+(1-2/q)_+}\right)^{|u|}|u|!\prod_{j\in u}\Upsilon_j^{\frac{1}{r+\alpha_p+1/2}}.
\end{align*}
By \cite[Theorem~2]{Pan2026Sharp}, the sum of these quantities over all finite subsets \(u\subseteq\mathbb N\) is finite. Hence,
 $\psi'_{r,\alpha,s,p,q,\gamma}(m)\leq C'_{\varepsilon,\Upsilon,r,\alpha,p,q} b^{\varepsilon m}$, with $C'_{\varepsilon,\Upsilon,r,\alpha,p,q}$ independent of $s$ and $m$.
\end{proof}

\bibliographystyle{siamplain}
\bibliography{references}
\end{document}